\documentclass[a4paper, reqno]{amsart}
\usepackage{anysize}
\usepackage{amssymb,latexsym}
\usepackage{amsmath,amsthm}
\usepackage{amsfonts,mathrsfs}
\usepackage{mathtools}

\usepackage{halloweenmath}

\usepackage{tikz-cd}

\usepackage{todonotes}

\usepackage[all]{xy}
\usepackage{graphicx}

\usepackage{xcolor,url}
\usepackage{hyperref}
\hypersetup{
    colorlinks,
    citecolor=blue,
    filecolor=blue,
    linkcolor=black,
    urlcolor=blue
}

\usepackage{enumitem}

\usepackage{listings}
\usepackage{color}

\definecolor{dkgreen}{rgb}{0,0.6,0}
\definecolor{gray}{rgb}{0.5,0.5,0.5}
\definecolor{mauve}{rgb}{0.58,0,0.82}

\usepackage{forest}

\usepackage{tabularray}
\UseTblrLibrary{diagbox}

\theoremstyle{plain}
\newtheorem{theorem}{Theorem}[section]

\newtheorem{corollary}[theorem]{Corollary}
\newtheorem{lemma}[theorem]{Lemma}

\newtheorem{fact}[theorem]{Fact}
\newtheorem{conjecture}[theorem]{Conjecture}
\newtheorem*{theorem*}{Theorem}
\newtheorem*{corollary*}{Corollary}
\newtheorem*{conjecture*}{Conjecture}

\theoremstyle{definition}
\newtheorem{definition}[theorem]{Definition}

\theoremstyle{remark}
\newtheorem{remark}[theorem]{Remark}
\newtheorem{claim}{Claim}

\newcommand\N{\mathbb{N}}

\newcommand\F{\mathbb{F}}

\def\cQ{\mathcal{Q}}

\def\cF{\mathcal{F}}

\newcommand\fS{\mathfrak{S}}

\newcommand\LL{\mathscr{L}}

\newcommand{\End}{\mathrm{End}}

\newcommand{\Id}{\mathrm{Id}}

\newcommand{\Asym}{\mathrm{Asym}}

\def\seq{\subseteq}

\newcommand{\set}[1]{\{ {#1} \}}
\newcommand{\vect}[1]{\langle {#1} \rangle}
\newcommand{\abs}[1]{\lvert {#1} \rvert}

\DeclareMathOperator{\Aut}{Aut}

\DeclareMathOperator{\ad}{ad}
\DeclareMathOperator{\Span}{Span}

\definecolor{airforceblue}{rgb}{0.36, 0.54, 0.66}

\def\Ind{\setbox0=\hbox{$x$}\kern\wd0\hbox to 0pt{\hss$\mid$\hss}
\lower.9\ht0\hbox to 0pt{\hss$\smile$\hss}\kern\wd0}
\def\Notind{\setbox0=\hbox{$x$}\kern\wd0\hbox to 0pt{\mathchardef
\nn=12854\hss$\nn$\kern1.4\wd0\hss}\hbox to
0pt{\hss$\mid$\hss}\lower.9\ht0 \hbox to 0pt{\hss$\smile$\hss}\kern\wd0}

\def\indi#1{\mathop{\ \ \hbox to 0ex{\hss$\vert^{\hbox to 0ex{$\scriptstyle#1$\hss}}$\hss}
\lower1ex\hbox to 0ex{\hss$\smile$\hss}\ \ }}

\def\nindi#1{\mathop{\ \ \hbox to 0ex{\hss$\!\not{\vert}^{\hbox to 0ex{$\scriptstyle\,#1$\hss}}$\hss}
\lower1ex\hbox to 0ex{\hss$\smile$\hss}\ \ }}

\begin{document}

\title{Wilson conjecture for omega-categorical Lie algebras, the case $4$-Engel characteristic $3$}

\author[C. d'Elb\'{e}e]{Christian d\textquoteright Elb\'ee$^\dagger$}
\address{School of Mathematics, University of Leeds\\
Office 10.17f LS2 9JT, Leeds}
\email{C.M.B.J.dElbee@leeds.ac.uk}
\urladdr{\href{http://choum.net/\textasciitilde chris/page\textunderscore perso/}{http://choum.net/\textasciitilde chris/page\textunderscore perso/}}

\thanks{
\begin{minipage}{0.8\textwidth}
 The author is fully supported by the UKRI Horizon Europe Guarantee Scheme, grant no EP/Y027833/1.
\end{minipage}%
\begin{minipage}{0.2\textwidth}
\begin{center}
    \includegraphics[scale=.04]{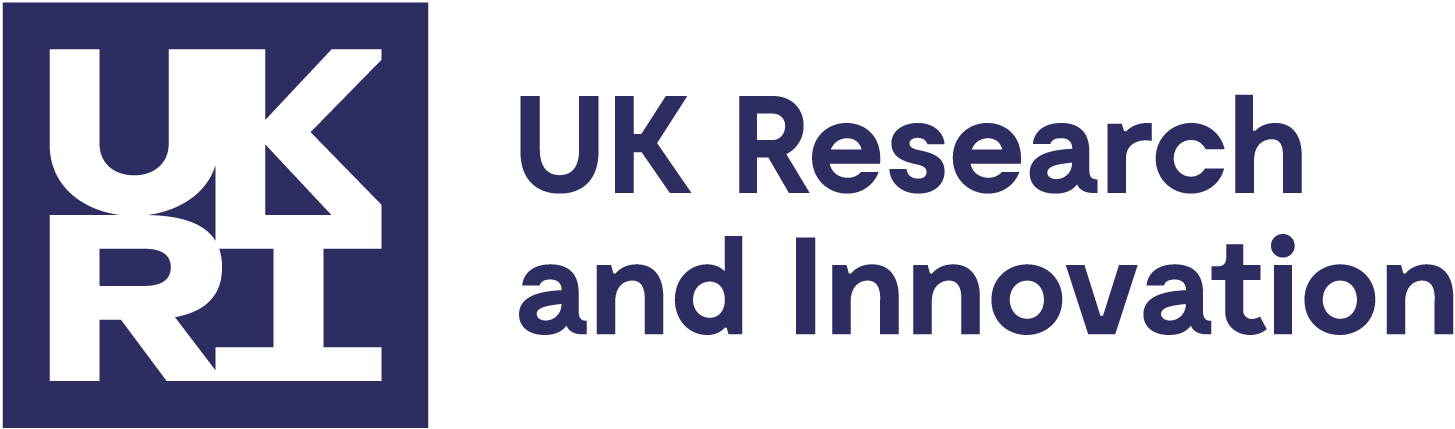}
\end{center}
\end{minipage}
}

\date{\today}

\begin{abstract}
We continue our study of the Wilson conjecture for $\omega$-categorical Lie algebras and prove that $\omega$-categorical $4$-Engel Lie algebras of characteristic $3$ are nilpotent. We develop a set of tools to adapt in the definable context some classical methods for studying Engel Lie algebras (Higgins, Kostrikin, Zelmanov, Vaughan-Lee, Traustason and others). We solve the case at hand by starting a systematic study of Lie algebras for which there is a $k$ such that the principal ideal generated by any element is nilpotent of class $<k$ (which we call \textit{$k$-strong Lie algebras}). We use computer algebra to check basic cases of a conjectural arithmetical property of those, namely that $x^{k-1}y^{k-1} = (-1)^{k-1}y^{k-1}x^{k-1}$ is an identity for Lie elements of the enveloping algebra. The solution is given by reducing the problem to $k$-strong Lie algebras generated by particularly well behaved \textit{sandwiches} in the sense of Kostrikin.
\end{abstract}

\maketitle

% \tableofcontents

\section{Introduction} 

We continue our study \cite{delbee2024engel3} of the Wilson conjecture for $\omega$-categorical Lie algebras, which is the following statement.
\begin{center}
\textit{Every $\omega$-categorical $n$-Engel Lie algebra over $\F_p$ is nilpotent.}
\end{center}
Recall that an $\omega$-categorical structure $M$ (group, Lie algebra, associative algebra) is one which has a unique countable model of its first-order theory, up to isomorphism. Equivalently, for each $n\in \N$, there are only finitely many orbits in the componentwise action of $\Aut(M)$ on the cartesian power $M^n$. A Lie algebra is $n$-Engel if it satisfies the identity 
\[0 = [\ldots [x,\underbrace{y],y],\ldots,y]}_{n \text{ times}} =: [x,y^n].
\]

The above statement is the Lie algebra analog of the original conjecture of Wilson \cite{wilson1981}:

\begin{conjecture*}[Wilson, 1981]\label{conj:wilson}
    Every locally nilpotent $\omega$-categorical $p$-group is nilpotent.
\end{conjecture*}

If solved, the Wilson conjecture would yield a quite satisfactory decomposition of $\omega$-categorical groups.  For a recent introduction and a general overview of the Wilson conjecture and its analogs, as well as the solution for $3$-Engel Lie algebras over $\F_5$, see \cite{delbee2024engel3}. Whereas the Wilson conjecture is wide open for groups and Lie algebras, the analogous question for associative algebra was solved by Cherlin \cite{cherlinnilringI,cherlinnilringII}, answering a question asked earlier by Baldwin and Rose \cite{baldwinrose1977}.

\begin{fact}[Cherlin, 1980]\label{fact:cherlinomegacatnilringarenilpotent}
    Any $\omega$-categorical nilring is nilpotent.
\end{fact}

Using a classical result of Zelmanov, the Wilson conjecture for Lie algebras is known to hold \textit{asymptotically} in the sense that for each given $n$, Lie algebras over $\F_p$ which are $n$-Engel are nilpotent for $p$ large enough (without the assumption of $\omega$-categoricity). The question of global nilpotency for small values of $n$ and $p$ is known up to $n = 5$, and is shown in the table below. As it appears, for small values of $n$ and $p$ such as $(n,p) \in \set{ (3,2), (3,5), (4,2), (4,3), (4,5)}$ the problem is then to prove that adding the property of $\omega$-categoricity allows to conclude nilpotency. As mentioned above, $3$-Engel Lie algebras of characteristic $5$ have been dealt with in \cite{delbee2024engel3}, and the solution uses almost no machinery. It is based on fortunate arithmetical properties of $3$-Engel Lie algebras of characteristic $5$, used to emulate a form of coding that appear in the work of Cherlin \cite{cherlinnilringII}. For the case at hand (namely $n = 4, p = 3$), a completely different approach has been undertaken, based on developing and adapting in the definable context, the tools that appear classically in the study of Engel Lie algebra, in the work of Higgins, Kostrikin, Zelmanov, Vaughan-Lee, Traustason and others. It is quite rare in model theory, a discipline which tends to produce very general results, to have such a dependency to parameters.

The number at the intersection of the $n$-Engel row and the characteristic $p$ column is a sharp bound on the nilpotency class of $n$-Engel Lie algebras of characteristic $p$. If that number is $\infty$, we mean that $n$-Engel Lie algebras of characteristic $p$ are not nilpotent in general.

\begin{table}[h]
\centering
\begin{tblr}{
  cells = {c},
  vline{1-2,7} = {-}{},
  hline{1-2,6} = {-}{},
}
\diagbox{$n$-Engel Lie algebra~~}{Characteristic $p$~~} & $2$                     & $3$                     & $5$                     & $7$                     & $>7$  \\
$2$-Engel                                                  & $2$                     & $3$                     & $2$ & $2$                     & $2$\\
$3$-Engel                                                  & $\infty$ & $4$                     & $\infty$ & $4$                     & $4$  \\
$4$-Engel                                                  & $\infty$ & $\infty$ & $\infty$ & $7$                     & $7$  \\
$5$-Engel                                                  & $\infty$ & $\infty$ & $\infty$ & $\infty$ & $11$ 
\end{tblr}
\end{table}The results for $2,3$ and $4$-Engel are due to Higgins \cite{higginsEngel}, Kostrikin \cite{kostrikinAroundBurnside} and Traustason \cite{TraustasonEngel3Engel41993, TraustasonEngel51996, Traustason3-4Engelcharacteristic21995} for exact bounds. The results for $5$-Engel Lie algebras are very recent and due to Vaughan-Lee \cite{vaughanlee20245engelliealgebras}.

\subsection{Strong Lie algebras} At the core of our approach is the notion of a \textit{$k$-strong Lie algebra}, which is a Lie algebra $L$ in which for every $x\in L$, the ideal $I(x)$ generated by $x$ satisfies $I(x)^k = 0$. In other words, $I(x)$ is nilpotent of class $<k$. Our terminology is new, and although this notion is omnipresent in the study of Engel Lie algebras, the author is not aware of a particular terminology for such Lie algebras. Clearly, a $k$-strong Lie algebra is $n$-Engel for $n = k$, and the converse statement is the one of interest. To our knowledge, this notion appears first in the work of Higman \cite{Higman1956} where the author proves that the associated Lie ring of a group of exponent $5$ -which is $4$-Engel- is $k$-strong for some $k$ (the argument is known as Higman algebra, see \cite{VaughanLee2011LiemethodsEngelgroups}) offering a solution to the restricted Burnside problem for groups of exponent $5$. After Zelmanov's celebrated work \cite{Zelmanov1990,Zelmanov1991}, it was known that $n$-Engel Lie algebras are locally nilpotent. More precisely, Zelmanov proves that there exists a function $f = f_{n,p}:\N \to \N$ such that every $r$-generated $n$-Engel Lie algebra over $\F_p$ is nilpotent of class $f(r)$. In a quest to estimate the growth rate of the functions $f_{n,p}$, the property of being a $k$-strong Lie algebra was revisited by the work of Havas, Newman and Vaughan-Lee \cite{HavasNewmanVaughanlee1990}, Traustason \cite{TraustasonEngel3Engel41993} and again very recently by Vaughan-Lee \cite{vaughanlee20245engelliealgebrasii} because of the following observation: a $k$-strong Lie algebra generated by $r$ elements is nilpotent of class $\leq kr$. In other words, $k$-strong Lie algebras constitute a class of $n$-Engel Lie algebras for which the nilpotency function $f_{n,p}$ grows \textit{linearly}.

There are examples of $n$-Engel Lie algebras for which the growth is not linear, in fact, which are not $k$-strong for any $k$. In \cite{TraustasonEngel3Engel41993}, Traustason proved that there exists a $3$-Engel Lie algebra over $\F_2$ which is not $k$-strong for any $k$, although the growth rate of the nilpotency function in that case is polynomial. Very recently, Vaughan-Lee \cite{vaughanlee20245engelliealgebrasii} proved an analogous result for some $5$-Engel Lie algebra over $\F_3$.  However, no example exists when $n<p$ and Zelmanov conjectures that none should exist.

\begin{conjecture}[Zelmanov, \cite{Zelmanov95Lieringmethods}]
For $n<p$, every $n$-Engel Lie algebra over $\F_p$ is $k$-strong for some $k$.
\end{conjecture}

Regardless of the veracity of Zelmanov's conjecture, it turns out, using a classical result of Kostrikin \cite{kostrikinAroundBurnside}, that the Wilson conjecture for $\omega$-categorical $n$-Engel Lie algebras over $\F_p$ with $n<p$ reduces to proving that all $\omega$-categorical $k$-strong Lie algebras are nilpotent.

We do a systematic study of $k$-strong Lie algebras, which are locally nilpotent almost by definition but still worth investigating. We conjecture that every $k$-strong Lie algebra satisfies the following identity (see Conjecture \ref{conj:toasties1}):
\[ x^{k-1}y^{k-1} = (-1)^{k-1}y^{k-1}x^{k-1}.\]
 We check that the above identity holds for $k\leq 5$ using the GAP package ModIsom \cite{ModIsom} which provides a neat nilpotent quotient algorithm for associative algebras. The method we use to check the conjecture is quite laborious and unsatisfactory: we deduce by hand a set of equations (in the envelope algebra) that follows from the condition $I(x)^k = 0$ and then compute a nilpotent quotient of a rank $2$, finitely presented associative algebra where the relators are given by this set of equations. For $k = 4$, we need $14$ equations (Theorem \ref{thm:toastiesconjecture4}) and for $k = 5$ we need $27$ equations (Theorem \ref{thm:toastiesconjecture5}). As suggested by Vaughan-Lee, an interesting lead would be to implement a nilpotent quotient algorithm to compute directly in $r$-generated $k$-strong Lie algebras, which should run smoothly for small values of $r$ and $k$. 

Our main result is the following.

\begin{theorem*}[\ref{thm:4strongLAnilpotent}]
    Every $\omega$-categorical $4$-strong Lie algebra over a field of characteristic $\neq 2$ is nilpotent.
\end{theorem*}

 The proof consist in making use of the identity above to interpret a commutative subalgebra of the enveloping algebra and then we applying the result of Cherlin (Fact \ref{fact:cherlinomegacatnilringarenilpotent}). By a result of Traustason \cite{TraustasonEngel3Engel41993}, every $4$-Engel Lie algebra of characteristic $3$ is $4$-strong, hence we deduce the case at hand.
\begin{corollary*}[\ref{thm:casen=4p=3}]
    Every $\omega$-categorical $4$-Engel Lie algebra over $\F_3$ is nilpotent.
\end{corollary*}

We also deduce that $\omega$-categorical $5$-strong Lie algebras of characteristic $p\neq 2,5$ are nilpotent (Corollary \ref{cor:5strongomegacategoricalnilpotent}).

\subsection{Toastie algebras} In a Lie algebra, we call \textit{toastie} an element which generates an abelian ideal, which corresponds -in the terms of Kostrikin- to a \textit{sandwich of infinite thickness}. Lie algebras generated by toasties -\textit{toastie algebras}- are used here as building blocks for decomposing $\omega$-categorical $k$-strong Lie algebras. The main quality of toastie algebras lies in the fact that identities are automatically equivalent to an associated \textit{linearized} version, regardless of the characteristic of the field. To exemplify this, it is standard that the $n$-Engel condition
\[
[x,y^n] = 0 
\]
implies its linearization
\[
\sum_{\sigma\in \fS_n} [x,y_{\sigma(1)},y_{\sigma(2)},\ldots, y_{\sigma(n)}]=0
\]
but only if $n<p$ do we have the converse implication in general. In a toastie algebra, every element is a sum of toasties and every identity is equivalent to a linearized one, without any dependence to the characteristic of the field. This implies the following crucial lemmas in our proof.
\begin{enumerate}
    \item The extension of scalars of a $n$-Engel toastie algebra is again $n$-Engel (and the same holds for $k$-strong).
    \item In a toastie algebra, the vector span of the set of values of a given polynomial is an ideal.
\end{enumerate} 
Those two facts should be familiar to anyone who has seen Kostrikin or Zelmanov's proof of the restricted Burnside problem. The trick used by Kostrikin and Zelmanov in order to be able to reach this point is to extend the scalars to an arbitrary infinite field, an operation which would not preserve $\omega$-categoricity. Similarly, the interplay between nonlinear and linear identities is omnipresent in the work of Zelmanov. The usual trick to ensure this interplay is to tensor out by an infinite Grassman algebra, an operation which is not permitted in our case, as it would also break $\omega$-categoricity. Toastie algebras are used here to circumvent the fact that, to remain in the category of $\omega$-categorical Lie algebras, few algebraic operations are permitted.

\subsection{Solvability vs nilpotency} One of the reasons why $n$-Engel Lie algebras in characteristic $p$ with $n<p$ are better understood comes from the classical fact of Higgins \cite{higginsEngel} that for these, solvability and nilpotency are equivalent notions. In \cite{wilson1981}, Wilson proves using the result of Cherlin that this is also true for locally nilpotent $\omega$-categorical groups, and we prove here by the same method that this is also true for $\omega$-categorical Engel Lie algebras (Corollary \ref{cor:wilsontheorem}). This fact is crucial for our solution of the case at hand.

\subsection{Conclusion} The tools developed here only allow us to conclude for the very particular case of $\omega$-categorical $4$-strong Lie algebras of characteristic $\neq 2$ hence for $4$-Engel Lie algebras of characteristic $3$. We also apply those methods for further reductions of the Wilson conjecture for $n$-Engel Lie algebras in characteristic $p>n$, see Theorem \ref{thm:furtherreduction}. However, those reductions are not quite satisfactory as they rely on our conjecture for $k$-strong Lie algebras. Nonetheless, this work sheds a light on unexploited connections between the methods developed for the solution of the restricted Burnside problem and the model theory involved in the Wilson conjecture for $\omega$-categorical Lie algebras. In the long run, we hope that the approach developed here can serve as a background to tackle more general cases, such as the case $n<p$. The case of $4$-Engel Lie algebras in characteristic $5$ is still resisting this approach and is a good test question. The case of characteristic $2$ stays quite impervious to those methods.

\subsection*{Acknowledgement} This paper is dedicated to Dugald Macpherson, for his 30 years of service at the University of Leeds, and for his unwavering support and friendship. I am still very thankful to Gunnar Traustason for his PhD thesis. I am very grateful to Michael Vaughan-Lee for the many runs of his algorithm on my behalf, and I am sorry that I still haven't figured out how to run his code by myself. I am very grateful to Laurent Bartholdi for his nilpotent quotient algorithm run for checking the $k$-strong conjecture for $k = 4$ and noticing an oddity in characteristic $2$. As I was trying to learn how to compute in GAP, I encountered a very dedicated and helpful community of people, and I want to give a special thanks to Eamon O'Brian (anupq package \cite{ANUPQ}) and Willem de Graaf (LieRing package \cite{LieRing}). My last thanks are directed to Guy, gamekeeper of the Ruantallain estate, for maintaining the bothy that sheltered me for several
nights on the western part of the Loch Tarbert, Jura, and from where most of this paper was written.

\section{Preliminaries}
\subsection{Notations} In a Lie algebra $L = (L,+,[.,.])$ over a field $\F$ we use the left-normed bracket notation, defined inductively by $[a_1,\ldots,a_n] = [[a_1,\ldots,a_{n-1}],a_n]$. We also often remove the commas $[a_1\ldots a_n] = [a_1,\ldots,a_n]$. When $a_2 = \ldots = a_n = b$ we simply write $[a_1\ldots a_n] = [ab^{n-1}]$. For $S,R\seq L$ we denote by $[S,R]$ the vector span of $\set{[a,b]\mid a\in L,b\in R}$ and iteratively $[S,R^n] = [R,S,\ldots,S] = [[R,\ldots, S],S]$. We denote by $I(S)$ the ideal in $L$ generated by $S$, i.e. the vector space
\[I(S) = \Span_\F(S)+[S,L]+[S,L,L]+\ldots\]
$I(S)$ is in particular a Lie subalgebra of $L$. A product of the form $abc$ (without brackets) will always stand for a product in an associative algebra, whereas $[abc] = [[a,b,],c]$ (with brackets) is always a product in a Lie algebra.

\subsection{Enveloping algebra} Let $L$ be an algebra over a field $\F$. Let $\End(L)$ be the $\F$ associative algebra of linear endomorphisms of the vector space $L$. For every $a\in L$ the map $\ad(a) : x\mapsto [x,a]$ is an endomorphism of $L$, i.e. an element of $\End(L)$. We define $A(L)$ as the subalgebra of $\End(L)$ generated by $\ad(L)$ i.e. all the endomorphisms $\ad(a)$ for $a\in L$. Elements of $A(L)$ which are equal to $\ad(a)$ for some $a\in L$ are called Lie elements of $A(L)$ and we will most of the time identify an element $a$ of $L$ with its adjoint $\ad(a)\in A(L)$. If $a,b\in L$ are seen as elements of $A(L)$, then the commutator $[a,b] = ab-ba$ (in the associative algebra sense) is the adjoint of the bracket $[a,b]\in L$, since by the Jacobi identity we have $[c,[a,b]] = [c,a,b] - [c,b,a]$, and we identify $[c,a,b]$ and $[c,b,a]$ with the elements $\ad(b)(\ad(a)(c))$ and $\ad(a)((\ad(b)(c))$ respectively.

\subsection{Caveat} There is an intrinsic ambiguity in our convention: in $A(L)$, the element $[ab^2]$ could stand for either $[[a,b],b] = [ab-ba,b] = ab^2-2bab+b^2a$ or for the commutator of $a\in L$ and $b^2\in A(L)$ which would be $[a,b^2] = ab^2-b^2a$. We will \textbf{never} mean the latter, as we will only consider commutators of Lie elements, and those always result in Lie elements. In other words, in $A(L)$, a commutator expression $[a,b,...]$ is always meant to result in a Lie element of $A(L)$.

\subsection{Extension of scalars}\label{subsec:interpretation} Let $L$ be a Lie algebra over a field of positive characteristic $p$. From a model theoretic point of view, we will consider $L$ in the language of Lie rings $\LL = \set{+,[.,.],0}$, which naturally encompasses the $\F_p$-vector space structure. Let $F$ be a finite field extension of $\F_p$. We argue that the extension of scalars 
\[\tilde L = L\otimes_{\F_p} F\]
is $0$-definable in $L$. Recall that the bracket $[\![.,.]\!]$ in the Lie algebra $\tilde L$ is obtained by extending linearly the expression:
\[[\![a\otimes \alpha, b\otimes \beta]\!] = [a,b]\otimes (\lambda\mu)\]
where $a,b\in L$ and $\alpha,\beta\in F$. Let $e_1,\ldots,e_d$ be an $\F_p$-basis of $F$. Write $e_ie_j = \sum_{k = 1}^d \lambda_k^{i,j} e_k$ for $\lambda_k^{i,j}\in \F_p$ and define $M_k$ to be the matrix $(\lambda_k^{i,j})_{1\leq i,j\leq d}$. Then one can write the multiplication in the field $F$ identified with $\F_p\times\ldots\times \F_p$ as:
\[\vec \alpha\cdot \vec \beta = (\alpha_1,\ldots,\alpha_d)\cdot(\beta_1,\ldots,\beta_d) = (\vec \alpha M_1\vec \beta,\ldots,\vec \alpha M_d \vec \beta)\]
where $\vec \alpha M_k\vec \beta = \sum_{1\leq i,j\leq d} \alpha_i\lambda_k^{i,j}\beta_j$. 
Now, every element of $\tilde L$ can be uniquely written as $\sum_i a_i\otimes \alpha_i$ for $a_i\in L$ and $\alpha_i\in F$, hence every element of $\tilde L$ can be uniquely written as a sum:
\[a_1\otimes e_1+\ldots+a_d\otimes e_d\]
for $a_1,\ldots,a_d\in L$. It follows that the bracket of two elements $\sum_{i = 1}^d a_i\otimes e_i, \sum_{i = 1}^d b_i\otimes e_i$ of $\tilde L$ is given by the formula:
\begin{align*}
    [\![\sum_{i = 1}^d a_i\otimes e_i, \sum_{i = 1}^d b_i\otimes e_i]\!]
               &= \sum_{1\leq i,j\leq d} [a_i,b_j]\otimes e_ie_j\\
               &= \sum_{i,j} [a_i,b_j]\otimes (\sum_k \lambda_k^{i,j} e_k)\\
               &= \sum_{1\leq i,j,k\leq d} \lambda_k^{i,j} [a_i,b_j]\otimes e_k.
\end{align*}

It should now be clear how to get a first-order definition of $\tilde L$ in the structure $(L,+,[.,.],0)$. Let $V$ be the cartesian product $L^d$ for which we define the new addition componentwise, and the new $0$ to be $(0,\ldots,0)$. We then define the new bracket of two elements $(a_1,\ldots,a_d),(b_1,\ldots,b_d)\in V$ via the formula
\[ [\![ (a_1,\ldots,a_d),(b_1,\ldots,b_d) ]\!] = (\sum_{1\leq i,j\leq d} \lambda_1^{i,j} [a_i,b_j],\sum_{1\leq i,j\leq d} \lambda_2^{i,j} [a_i,b_j],\ldots,\sum_{1\leq i,j\leq d} \lambda_d^{i,j} [a_i,b_j]).\]
By the above, the structure $(V,+,[\![.,.]\!],0)$ is isomorphic to $\tilde L$. 

Of course, nothing here is really specific to Lie algebras, and given any algebra $A$ over a field $F$ and a finite-dimensional algebra $B$ over $F$, the algebra $A\otimes _F B$ is definable in the $F$-algebra structure $A$.

\section{Interpretation in the enveloping algebra}
Let $L$ be a Lie algebra, $\ad:L\to A(L)$ the adjoint representation, where $A(L)$ is the algebra of endomorphisms generated by $\ad(L)$.

Given a subalgebra $B$ of $A(L)$ and a fixed element $a\in L$ such that for all $f,g\in B$, we have $f\circ g(a) = g\circ f(a)$, we denote by $B_a$ the associative algebra defined by the following:
\begin{itemize}
    \item The domain of $B_a$ is given by the quotient of $B$ by the equivalence relation defined by $f\sim g$ if and only if $f(a) = g(a)$.
    \item Denote by $\tilde f, \tilde g\in B_a$ the classes of $f,g\in B$ modulo $\sim$. We define the sum $\tilde f +\tilde g$ to be the class of $(f+g)\in B$ modulo $\sim$. It is immediate that this addition is well-defined on $B_a$.
    \item For $\tilde f,\tilde g\in B_a$, we define the multiplication $\tilde f \tilde g$ to be the class of $f\circ g$ modulo $\sim$. Here it is crucial that $f\circ g(a) = g\circ f(a)$ in order for the multiplication in $B_a$ to be well-defined. To see this, assume $f\sim f'$ and $g\sim g'$, then to get $f\circ g \sim f'\circ g'$ we compute: $f\circ g(a) = f(g(a)) = f(g'(a)) = g'(f(a)) = g'(f'(a)) = f'(g'(a)) = f'\circ g'(a)$.
\end{itemize}

A nonconstant (associative, noncommutative) monomial $p(x_1,\ldots,x_s)$ over $\F_p$ is an expression of the form $\lambda x_{i_1}^{\alpha_1}\cdots x_{i_k}^{\alpha^k}$ with $\alpha_i\geq 1$, $\lambda\in \F_p$. In any Lie algebra $L$ over $\F_p$, if $ b_1,\ldots,b_s\in L$, then $p(b_1,\ldots,b_s)$ is an element of $A(L)$ and for all $a\in L$, the expression $[ap(b_1,\ldots,b_s)] = \lambda [ab_{i_1}^{\alpha_1}\cdots b_{i_k}^{\alpha^k}]$ is a well-defined element of $L$.
\begin{theorem}\label{thm:interpretationBa}
    Let $L$ be an $\omega$-categorical Lie algebra over $\F_p$. Let $\cF$ be a (possibly infinite) family of nonconstant monomials over $\F_p$ and $X\seq A(L)$ be given by 
    \[X = \set{p(\ad(b_1),\ldots,\ad(b_s))\mid b_1,\ldots,b_s\in L, p(x_1,\ldots,x_s)\in \cF}
    \]
    Let $B\seq A(L)$ be the associative algebra generated by $X$ and $a\in L$ an element such that $f\circ g(a) = g\circ f(a)$ for all $f,g\in B$. Then the associative algebra $B_a$ is interpretable in $L$.
\end{theorem}
\begin{proof}
    We fix $\cF,X,B$ as in the hypotheses. 
    \begin{claim}\label{claim:interpretationreduction1}
        There exists $M\in \N$ and a finite subset $\cF_0\seq \cF$ such that for all $h\in B$ and $a\in L$ there exists $\vec z_1,\ldots,\vec z_M$ such that $ \abs{\vec z_i} = k_i\leq M$ and $\vec z_i = (p_{1,i}(\vec t_{1,i}),\ldots,p_{k_i,i}(\vec t_{k_i,i}))$ for $p_{i,j}\in \cF_0$ and such that 
        \[h(a) = \sum_{i = 1}^M [a,p_{1,i}(\vec b_{1,i}),\ldots,p_{k_i,i}(\vec b_{k_i,i})]\]
        for tuples $\vec b_{i,j}$ of elements of $L$ with $\abs{\vec b_{i,j}} = \abs{\vec t_{i,j}}$.
    \end{claim}
    \begin{proof}
         For $n\in \N$, $k_1,\ldots, k_n\leq n$, and $p_{1,1}(\vec t_{1,1}),\ldots,p_{k_1,1}(\vec t_{k_1,1}),\ldots, p_{1,n}(\vec t_{1,n}),\ldots, p_{k_n,n}(\vec t_{k_n,n})\in \cF$, consider the $0$-definable set $A(n,k_1,\ldots,k_n, (p_{j,i})_{i\leq n,j\leq k_i})\seq L\times L$ given by 
        \[\set{(x,y)\mid \exists \vec t_{1,1}\ldots \exists \vec t_{k_n,n}\ y = \sum_{i = 1}^n [x,p_{1,i}(\vec t_{1,i}),\ldots,p_{k_i,i}(\vec t_{k_i,i})]}.\]
        Each $A(n,k_1,\ldots,k_n, (p_{j,i})_{i\leq n,j\leq k_i})$ is invariant under automorphisms and so is the union $A$ ranging over all data $n\in \N$, $k_1,\ldots, k_n\leq n$, $(p_{j,i})_{i\leq n,j\leq k_i}\seq \cF$. It follows from $\omega$-categoricity that $A$ is a union of finitely many sets of the form $A(n,k_1,\ldots,k_n, (p_{j,i})_{i\leq n,j\leq k_i})$. Set $M$ to be the maximum of those finitely many $n$'s and $\cF_0$ the union of those finitely many $(p_{j}^i)$. As any element of $B$ is given by a map of the form $x\mapsto \sum_{i = 1}^n [x,f_1^i,\ldots,f_{s_i}^i]$, for some $n\in \N$ and $f_j^i\in X$, the pair $(a,h(a))$ belongs to $A$ hence the conclusion of the claim follows. Note that one can always set one of the $[x,p_{1,i}(\vec t_{1,i}),\ldots,p_{k_i,i}(\vec t_{k_i,i})]$ to be zero because the $p_{i,j}$ are nonconstant monomials.
    \end{proof}

    \begin{claim}\label{claim:interpretationreduction2}
        There exists $N\in \N$ and nonconstant monomials $q_1(\vec t_{1}),\ldots,q_{N}(\vec t_N)$ which are products of elements of $\cF$ with $\abs{\vec t_i}= N$ such that for all $h\in B$ and $a\in L$ there exists $\vec b_1,\ldots,\vec b_N$ with $\abs{\vec b_i} = N$ and such that 
        \[h(a) = \sum_{i = 1}^N [aq_{i}(\vec b_i)]\]
    \end{claim}
    \begin{proof}
        With the same notations as in Claim \ref{claim:interpretationreduction1}, there are only finitely many products of the form $p_{1,i}(\vec t_{1,i}),\ldots,p_{k_i,i}(\vec t_{k_i,i})$ with $p_{i,j}(\vec t_{i,j})\in \cF_0$ and $k_i\leq M$. Call the set of all those products $\cQ_0$ and its cardinality $s\in \N$. Any element of $\cQ_0$ is a nonconstant monomial and let $m$ be the maximal arity of such monomial. Any such monomial can be considered as a monomial in $m$ variable. Let $N = \max\set{M,s,m}$. Of course, any element of $\cQ_0$ can be considered as a monomial in $N$ variables. Again, any nonconstant monomial vanishes when evaluated at $(0,\ldots,0)$, hence we may add dummy summand of the form $[x,q(\vec t)]$ with $q\in \cQ_0$ and new variables $\vec t$, or add dummy variables, to ensure $\abs{\vec t} = N$, so that the expression $\sum_{i = 1}^M [a,p_{1,i}(\vec b_{1,i}),\ldots,p_{k_i,i}(\vec b_{k_i,i})]$ equals an expression of the form $\sum_{i = 1}^N [aq_{i}(\vec c_i)]$.
    \end{proof}

    We fix $a\in L$ as in the hypotheses. Let $\sim$ be the equivalence relation on $B$ given by $f\sim g$ if and only if $f(a) = g(a)$. Let $U$ be the cartesian product $L^{N^2}$. For each element $b = (\vec b_1,\ldots, \vec b_N)\in U$, the map $f_b: x\mapsto \sum_{i = 1}^N [x,q_i(\vec b_i)]$ is an element of $B$ and by Claim \ref{claim:interpretationreduction2}, for any $h\in B$ there exists $b\in U$ such that $h\sim f_b$. We extend $\sim$ on elements of $U$ by setting $b\sim c$ if and only if $f_b\sim f_c$. Of course, the relation $\sim$ on $U$ is definable using the parameter $a$ by the following formula:
    \[b\sim c \iff \sum_{i = 1}^N [a,q_i(b_i)] = \sum_{i = 1}^N [a, q_i(c_i)]\]

    Let $V = U/\sim$. By the above, the map $b/\sim \to f_b/\sim$ is a bijection between $V$ and $B_a$. The addition on $V$ is defined as follows: from $b,c\in U$, the map $f_b+f_c$ is an element of $B$ hence by Claim \ref{claim:interpretationreduction2} there exists $d\in U$ such that $f_b+f_c \sim f_d$, so we define the sum $(b/\sim)+(c/\sim)$ to be $d/\sim\in V$.

    Similarly, the product of $b/\sim$ and $c/\sim$ is defined to be $d/\sim$ where $d\in U$ is such that $f_b\circ f_c\sim f_d$. Let us check that this product is well-defined. Assume that $b\sim b'$ and $c\sim c'$, that $d$ is such that $f_d\sim f_b\circ f_c$ and $d'$ is such that $f_{d'}\sim f_{b'}\circ f_{c'}$. We need to check that $d\sim d'$, i.e. $f_d\sim f_{d'}$. We compute
    \begin{align*}
        f_d(a) &= f_b\circ f_c(a)\\
        &= f_b(f_c(a))\\
        &= f_b(f_{c'}(a))\\
        &= f_{c'}(f_b(a))\\
        &= f_{c'}(f_{b'}(a))\\
        &= f_{b'}(f_{c'}(a))\\
        &= f_{d'}(a).
    \end{align*}
    Once again, the fact that $f\circ g(a) = g\circ f(a)$ for all $f,g\in B$ is crucial.
\end{proof}

We deduce an analog for Lie algebras of a result of Wilson, originally proved for $\omega$-categorical groups \cite{wilson1981}. The proof is in the same spirit as Wilson's original proof. In a Lie algebra $L$, we write $L^{(k)}$ for the $k$-th element of the derived series, defined by $L^{(1)} = L$ and $L^{(k+1)} = [L^{(k)},L^{(k)}]$. We use the following Lie algebra analog of a classical theorem of P. Hall. 

\begin{fact}\label{fact:chao}
    \cite[Theorem 2]{Chao1968} Let $L$ be a Lie algebra and $N$ a nilpotent ideal of $L$. Then $L$ is nilpotent if and only if $L/N^{(2)}$ is nilpotent.
\end{fact}

\begin{corollary}[Wilson's Theorem]\label{cor:wilsontheorem}
    Let $L$ be an $\omega$-categorical Engel Lie algebra. Then $L$ is solvable if and only if $L$ is nilpotent.
\end{corollary}
\begin{proof}
Only the left to right direction needs a proof. Assume that $L \supsetneq L^{(2)}\supsetneq L^{(3)}\ldots \supsetneq L^{(s)} = 0$. 
        Using Fact \ref{fact:chao}, by induction on the derived length, it is enough to prove that $L^{(s-2)}$ is nilpotent. So we may simply assume that $L^{(3)} = 0$ i.e. $ L^{(2)}$ is abelian.
 Let $a\in L'$, and $b,c\in L$, we have 
        \[[a,b,c] = [a,c,b]+[a,[b,c]] = [a,c,b]\]
in other words, for all $b,c\in L$ we have $\ad(b)\circ \ad(c) (a) = \ad(c)\circ \ad(b)(a)$. It follows that for $B = A(L)$, we have $f\circ g(a) = g\circ f(a)$ for all $f,g\in B$, and $B_a$ is a commutative associative algebra. It follows from Theorem \ref{thm:interpretationBa} that $B_a$ is $\omega$-categorical. As $L$ is uniformly locally nilpotent, $B$ and hence $B_a$ are nilrings so that $B_a$ is nilpotent by Cherlin's Theorem (Fact \ref{fact:cherlinomegacatnilringarenilpotent}).
 Then there exists $m\in \N$ such that $[a,b_1,\ldots,b_m] = 0$ for all $b_1,\ldots,b_m\in L$. As $a$ is arbitrary in $L'$ and for all $c_1,c_2$ we have $a = [c_1,c_2]\in L'$ we conclude that $L$ is $m+1$-nilpotent.
\end{proof}
The latter result is extremely useful as it allows to do inductions on series. The next corollary exemplifies this phenomenon.

\begin{remark}
    Corollary \ref{cor:wilsontheorem} is also true when the assumption ``$\omega$-categorical" is replaces by ``$n$-Engel Lie algebras over a field of characteristic $p$ with $n<p$". This is a classical result of Higgins \cite{higginsEngel}. 
\end{remark}

\begin{definition}
    In a Lie algebra $L$, an ideal is \textit{characteristic} if it is closed under automorphisms of $L$. We call a Lie algebra \textit{characteristically simple} if there are no nontrivial characteristic ideals in $L$. 
\end{definition}

\begin{remark}
    The reader should be careful that this is \textit{not} a standard terminology. In the literature on Lie algebras, a characteristic ideal is one which is closed under all derivations. Here we draw a direct parallel with the notion of a characteristic subgroup.
\end{remark}

\begin{corollary}
    Assume that every $\omega$-categorical characteristically simple Engel Lie algebra is nilpotent. Then every $\omega$-categorical Engel Lie algebra is nilpotent. 
\end{corollary}
\begin{proof}
    Assume the premises and let $L$ be any $\omega$-categorical Engel Lie algebra. Any characteristic ideal is a union of orbits under the automorphism group, hence there are only finitely many of those, and let 
    \[L = I_1\supsetneq I_2\supsetneq \ldots \supsetneq I_{s+1} = 0\]
    be a maximal chain. Each successive quotient $I_i/I_{i+1}$ is characteristically simple and Engel, hence nilpotent by hypothesis. As $I_{s}$ is nilpotent and $I_{s-1}/I_s$ is nilpotent, we have that $I_{s-1}$ is solvable, hence by Wilson's theorem (Corollary \ref{cor:wilsontheorem}) $I_{s-1}$ is nilpotent. An immediate induction yields that $L$ is nilpotent.
\end{proof}

\section{Solution for $n = 4$ and $p = 3$}

\subsection{Toastie and $n$-strong algebras} Following Kostrikin, a \textit{sandwich} of a Lie algebra $L$ is an element $c$ satisfying the following two equalities
\[[ac^2] = 0\text{ and }[acbc] = 0 \]
for all $a,b\in L$. If the ambient characteristic is not $2$ then the first equality implies the second. The \textit{thickness} of a sandwich $c$ is the smallest number $k$ such that for some $a,b_1,\ldots,b_{k+1}$ we have 
\[[acb_1\ldots b_{k+1}c]\neq 0\]
If there is no such $k$ the sandwich $c$ is called \textit{of infinite thickness}. It is easy to see that $c$ is a sandwich of infinite thickness if and only if the ideal $I(c)$ is abelian. Sandwiches of infinite thickness will be the main interest of this section, and we will prefer the shorter and more conceptual name of \textit{toasties}\footnote{Credits to Dugald Macpherson for that name. My question was: what is like a sandwich, but better? His answer was spontaneous.}.
\begin{definition}
    An element $c\in L$ that generates an abelian ideal is called a \textit{toastie}. A Lie algebra generated by toasties will be called a \textit{toastie algebra}.
\end{definition}
It is easily checked that the bracket of two toasties is again a toastie and the scalar multiple of a toastie is again a toastie. In a toastie algebra $L$ we denote by $S = S(L)$ the set of toasties of $L$, we have
\[L = \vect{S(L)} = \Span(S(L)).\]
In other words, in a toastie algebra, every element is a sum of toasties. It is immediate then that toastie algebras are locally nilpotent. Note that if  $[S(L),L]\seq S(L)$, the sum of two toasties is not a toastie in general. 
\begin{lemma}\label{lm:toastielinearspanengel}
    Let $L$ be a toastie algebra and $m\in \N$. Then the set $\Span([ab^m]\mid a,b\in L)$ is an ideal of $L$.
\end{lemma}
\begin{proof}
We start with a claim.
    \begin{claim}
        Let $a,d\in L$ and $c_1,\ldots,c_s\in S(L)$. Then \[[a(c_1+\ldots+c_s)^m d] = [ad(c_1+\ldots+c_s)^m]+[a([c_1,d]+c_2+\ldots+c_s)^m]+\ldots+[a(c_1+\ldots+[c_s,d])^m]\]
    \end{claim}
    \begin{proof}[Proof of the Claim]
    Observe first that in $A(L)$ 
 and for any toasties $c_1,\ldots c_s$ we have
    \[(c_1+\ldots+c_s)^m = \sum_{1\leq i_1\neq \ldots \neq i_m\leq s}c_{i_1}\cdots c_{i_m}\quad \quad (\star)\]
    where the sum is taken over every tuple $(i_1,\ldots,i_m)$ of distinct elements of $\set{1,\ldots,s}$. It follows that 
    \[(c_1+\ldots+c_s)^md = \sum_{1\leq i_1\neq \ldots \neq i_m\leq s}c_{i_1}\ldots c_{i_m} d.\]
    Fix $(i_1,\ldots,i_m)$. Iterating the use of the Jacobi identity, we obtain 
    \begin{align*}
        c_{i_1}\cdots c_{i_m} d &= c_{i_1}\cdots c_{i_{m-1}} dc_{i_m} + c_{i_1}\cdots c_{i_{m-1}} [c_{i_m}d]\\
        &= c_{i_1}\cdots c_{i_{m-2}} dc_{i_{m-1}}c_{i_m}+ c_{i_1}\cdots c_{i_{m-2}} [c_{i_{m-1}}d]c_{i_m}+c_{i_1}\cdots c_{i_{m-1}} [c_{i_m}d]\\
        &\vdots\\
        &= dc_{i_1}\cdots c_{i_m}+ [c_{i_1}d]c_{i_2}\cdots c_{i_m}+c_{i_1}[c_{i_2}d]c_{i_3}\cdots c_{i_m}+\ldots + c_{i_1}\cdots [c_{i_m}d]
    \end{align*}
    Then 
    \begin{align*}
        \sum_{1\leq i_1\neq \ldots \neq i_m\leq s}c_{i_1}\ldots c_{i_m} d = &\sum_{1\leq i_1\neq \ldots \neq i_m\leq s}dc_{i_1}\cdots c_{i_m}\\
        &+ \underbrace{\sum_{1\leq i_1\neq \ldots \neq i_m\leq s} [c_{i_1}d]c_{i_2}\cdots c_{i_m}+c_{i_1}[c_{i_2}d]c_{i_3}\cdots c_{i_m}+\ldots + c_{i_1}\cdots [c_{i_m}d]}_{= B}.
    \end{align*}
    By $(\star)$, we know that  $\sum_{1\leq i_1\neq \ldots \neq i_m\leq s}dc_{i_1}\cdots c_{i_m} = d(c_1+\ldots+c_s)^m$. We turn to $B$.
    For each $1\leq i\leq s$ we write $d_i = [c_i,d]$. Picking up all summands of $B$ which involve the element $d_1 = [c_1d]$, we easily see that every monomial of multiweight $(1,\ldots,1)$ in $d_1,c_{i_2},\ldots c_{i_m}$ for $1<i_2\neq \ldots \neq i_m\leq s$ appear once and only once in $B$. The sum of all those equals $(d_1+c_2+\ldots +c_s)^m$, using again $(\star)$. Similarly, every monomial of multiweight $(1,\ldots,1)$ in $c_{i_1},d_{2},\ldots c_{i_m}$ for $1\leq i_1\neq \ldots \neq i_m\leq s$ and $i_j\neq 2$ appears once and only once, and the sum of all those equals $(c_1+d_2+c_3\ldots +c_s)^m$. Iterating, we see that 
    \[B = (d_1+c_2+\ldots +c_s)^m+(c_1+d_2+c_3+\ldots +c_s)^m+\ldots +(c_1+\ldots c_{s-1}+d_s)^m\]
    and we conclude the claim.
    \end{proof}
Let $I = \Span([ab^m]\mid a,b\in L)$. To conclude, observe that it suffices to show that for all $a,b,d\in L$ we have $[ab^md]\in I$. Write $b = c_1+\ldots+c_s$ with $c_i\in S(L)$ and $b_i = c_1+\ldots+[c_i,d]+\ldots+c_s$. We have $[adb^m]\in I$ and $[ab_i^m]\in I$ for each $1\leq i\leq s$. By the claim, $[ab^md]\in I$.
\end{proof}

\begin{remark}\label{rk:moretedious}
    A similar but more tedious proof yields that in a toastie algebra $L$, if $p(x_1,\ldots,x_s)$ is a monomial, then the set $\Span([ap(b_1,\ldots,b_s)]\mid a,b_1,\ldots,b_s\in L)$ is an ideal of $L$.
\end{remark}

Let $L$ be an $\omega$-categorical toastie algebra. By considering the family of sets \[U_k = \set{a\in L\mid a = b_1+\ldots+b_k\text{ for some $b_1,\ldots,b_k \in S(L)$}}\]
there exists some $n\in \N$ such that every element of $L$ is a sum of $n$ toasties (note that $0$ is a toastie), which implies that for each $a\in L$ we have $I(a)^{n+1} = 0$. This motivates the following definition.

\begin{definition}
    A Lie algebra is $n$-strong if $I(a)^n = 0$ for each $a\in L$.
\end{definition}

\begin{lemma}\label{lm:linearizationEngeltoastie}
    Let $L$ be a toastie algebra, and let $S$ be a set of toasties generating $L$ (e.g. $S = S(L)$). Then 
    \begin{itemize}
        \item $L$ is $n$-Engel if and only if $\sum_{\sigma\in \fS_n} c_{\sigma(1)}\cdots c_{\sigma(n)} = 0$ for all $c_1,\ldots,c_n\in S$.
        \item $L$ is $n$-strong if and only if $\sum_{\sigma\in \fS_n} c_{\sigma(1)} w_1c_{\sigma(2)} \cdots c_{\sigma(n-1)} w_{n-1}c_{\sigma(n)} = 0$ for all $c_1,\ldots,c_n\in S$ and products of Lie elements $ w_1,\ldots, w_{n-1}\in A(L)$.
    \end{itemize}
\end{lemma}
\begin{proof}
If $L$ is $n$-Engel, let $c_1,\ldots,c_n\in S$, then $0 = (c_1+\ldots +c_n)^n = \sum_{\sigma\in \fS_n} c_{\sigma(1)}\cdots c_{\sigma(n)}$, so we prove the converse. Assume that $\sum_{\sigma\in \fS_n} c_{\sigma(1)}\cdots c_{\sigma(n)} = 0$ for all $c_i\in S$, and let $b\in L$. By hypothesis, there exists $c_1,\ldots,c_s\in S$ such that $b = c_1+\ldots+c_s$. We conclude 
    \[b^n = (c_1+\ldots+c_s)^n = \sum_{1\leq i_1\neq \ldots \neq i_n\leq s}c_{i_1}\cdots c_{i_n} = \sum_{\set{d_1,\ldots,d_n}\seq \set{c_1,\ldots,c_s}} \sum_{\sigma\in \fS_n} d_{\sigma(1)}\cdots d_{\sigma(n)} = 0.\]
For the second bullet, the reasoning is similar as for the first bullet, only more tedious to write down. Fix $c_1,\ldots,c_n\in S$ and products of Lie elements $ w_1,\ldots, w_{n-1}$ . As $I(c_1+\ldots+c_n)^n = 0$ and $(c_1+\ldots+c_n) w_i\in I(c_1+\ldots+c_n)$ for each $i$, we conclude that 
\[0 = (c_1+\ldots+c_n) w_1(c_1+\ldots+c_n) w_2\cdots w_{n-1}(c_1+\ldots+c_n) = \sum_{\sigma\in \fS_n} c_{\sigma(1)} w_1c_{\sigma(2)} \cdots c_{\sigma(n-1)}  w_{n-1}c_{\sigma(n)}\]
Conversely, assume that $\sum_{\sigma\in \fS_n} c_{\sigma(1)}  w_1c_{\sigma(2)} \cdots c_{\sigma(n-1)}  w_{n-1}c_{\sigma(n)} = 0$ for all $c_1,\ldots,c_n\in S$, and products of Lie elements $w_1,\ldots,  w_{n-1}\in A(L)$. Let $b = c_1+\ldots+c_s$ with $c_i\in S$. In order to prove that $I(b)^n = 0$, it is enough to prove that for all products of Lie elements $  w_1,\ldots,  w_{n-1}\in A(L)$, we have $[b  w_1 b\ldots   w_{n-1}b] = 0$. Similarly as above:
\begin{align*}
    [b  w_1 b\ldots   w_{n-1}b] &= [(c_1+\ldots+c_s)  w_1 (c_1+\ldots+c_s)\ldots   w_{n-1}(c_1+\ldots+c_s)]\\
    &= \sum_{1\leq i_1\neq \ldots \neq i_n\leq s}[c_{i_1}  w_1c_{i_2}\cdots   w_{n-1}c_{i_n}]\\
    & = \sum_{\set{d_1,\ldots,d_n}\seq \set{c_1,\ldots,c_s}} \sum_{\sigma\in \fS_n} [d_{\sigma(1)}  w_1d_{\sigma(2)}\cdots   w_{n-1}d_{\sigma(n)}]\\
    & = 0
\end{align*}
\end{proof}

\begin{remark}
    Note that $n$-Engel always implies the identity $\sum_{\sigma\in \fS_n} a_{\sigma(1)}\cdots a_{\sigma(n)} = 0$, hence for a toastie algebra, this is equivalent. It is well-known that the $n$-Engel condition is equivalent to the linearized condition $\sum_{\sigma\in \fS_n} a_{\sigma(1)}\cdots a_{\sigma(n)} = 0$ if the characteristic $p$ is greater than $n$, but this is not true in general.
\end{remark}

\begin{lemma}\label{lm:scalarextensiontoastie}
Let $L$ be a Lie algebra over a field $K$ and let $F$ be a field extension of $K$.
    \begin{enumerate}
        \item If $L$ is a toastie algebra, then so is $L\otimes_K F$.
        \item If $L$ is a toastie algebra which is $n$-Engel, then so is $L\otimes_K F$.
        \item If $L$ is an $n$-strong toastie algebra, then so is $L\otimes_K F$.
    \end{enumerate}    
\end{lemma}
\begin{proof}
    Let $\tilde L = L\otimes_K F$. First observe that the set $S = \set{c\otimes f\mid c\in S(L),f\in F}$ is a set of toasties of $\tilde L$. To see this, note that an element of $I(c\otimes f)$ is a sum of elements of the form $[c\otimes f,a_1\otimes f_1,\ldots,a_k\otimes f_k]$ which by definition equals $[ca_1\cdots a_k]\otimes ff_1\cdots f_k$. The bracket of two elements of $I(c\otimes f)$ can be written as a sum of bracket:
    \[[[ca_1\cdots a_k]\otimes ff_1\cdots f_k , [cb_1\cdots b_s]\otimes fg_1\cdots g_s] = [ca_1\cdots a_k[cb_1\cdots b_s]]\otimes ff_1\cdots f_k f g_1\cdots g_s.\]
    If $c$ is a toastie, all of those summands vanish and $I(c\otimes f)$ is abelian.
    
    By definition, every element of $\tilde L$ is a sum of basic tensors $a\otimes f$ with $a\in L$ and $f\in F$. As $L$ is a toastie algebra, $a$ is itself a sum of toasties of $L$ hence every element of $\tilde L$ is a sum of elements of $S$, hence $\tilde L$ is a toastie algebra, which yields $(1)$.

    The proof of $(2)$ is very similar to the one of $(3)$, only easier, hence we leave it as an exercise to the reader.

    We turn to $(3)$. Using Lemma \ref{lm:linearizationEngeltoastie}, it is enough to prove that for all $c_1\otimes f_1,\ldots,c_n\otimes f_n\in S$ and products of Lie elements $ w_1,\ldots, w_{n-1}\in A(\tilde L)$ we have \[\sum_{\sigma\in \fS_n} c_{\sigma(1)}\otimes f_{\sigma(1)}  w_1c_{\sigma(2)} \cdots c_{\sigma(n-1)}\otimes f_{\sigma(n-1)}  w_{n-1}c_{\sigma(n)}\otimes f_{\sigma(n)} = 0.\]
    First, as every element in $\tilde L$ is a sum of basic tensors, we may assume that $  w_i$ are products of basic tensors, i.e. $  w_i = a_{i,1}\otimes f_{i,1}\cdots a_{i,k_i}\otimes f_{i,k_i}$. Then by definition, $c_{\sigma(1)}\otimes f_{\sigma(1)}  w_1c_{\sigma(2)} \cdots c_{\sigma(n-1)}\otimes f_{\sigma(n-1)}  w_{n-1}c_{\sigma(n)}\otimes f_{\sigma(n)}$ equals the following
    \[ (c_{\sigma(1)}a_{1,1}\ldots a_{1,k_1}c_{\sigma(2)}\cdots a_{n-1,1}\ldots a_{n-1,k_{n-1}}c_{\sigma(n)})\otimes(f_{\sigma(1)}f_{1,1}\ldots f_{1,k_1}f_{\sigma(2)}\cdots f_{n-1,1}\ldots f_{n-1,k_{n-1}}f_{\sigma(n)}) \]
The sum of all those terms where $\sigma$ ranges in $\fS_n$ yields $0$ as $L$ is $n$-toastie, by Lemma \ref{lm:linearizationEngeltoastie}.
\end{proof}

\subsection{Identities in n-strong algebras} We now argue that $n$-strong algebras satisfy a form of commutativity. In fact, we propose the following conjectures.

\begin{conjecture}\label{conj:toasties1} Every $n$-strong Lie algebra over a field of size $\geq n$ satisfies 

\begin{enumerate} 
    \item[$(\mathrm{I})$] $a^{n-1}b^{n-1} = (-1)^{n-1}b^{n-1}a^{n-1}$ for all Lie elements $a,b\in A(L)$.
    \item[$(\mathrm{II})$]  $p(a,b) = (-1)^{n-1} p(b,a)$ for all monomial $p(x,y)$ of multiweight $(n-1,n-1)$ and Lie elements $a,b\in A(L)$.
    \item[$(\mathrm{III})$] $p(a_1,\ldots,a_s)$ and $p(a_{\sigma(1)},\ldots,a_{\sigma(s)})$ are $\F_p$-linearly dependent, for all monomials $p(x_1,\ldots, x_s)$ of multiweight $(n-1,n-1,\ldots,n-1)$, and Lie elements $a_i\in A(L)$.
\end{enumerate}
\end{conjecture}

\begin{remark}
    Version $(\mathrm{I})$ of the conjecture (for $n=4$) will be at the core of our solution to Wilson's conjecture for $4$-Engel Lie algebras of characteristic $3$. Version $(\mathrm{II})$ is what we are generally able to prove for small values of $n$, and version $(\mathrm{III})$ is merely a natural generalization. Note that the requirement that the field has size $\geq n$ might be superfluous. 
\end{remark}

We will mainly focus on version (I) of the conjecture. Let us check it for $n = 2$. Note that any $n$-strong Lie algebra is also $n$-Engel, and here the identity $ab = -ba$ holds for any $2$-Engel Lie algebras $L$ since, in $A(L)$ we have
\[0 = (a+b)^2 = ab+ba\]
for all $a,b\in L$.

Let us check the conjecture (I) for $n = 3$, so assume now that $L$ is $3$-strong. We will use the method of linearization, which involves hypotheses on the size of the underlying field. For arbitrary $a,b,c$ and scalar $\lambda$, as $I(a+\lambda b)^3 = 0$, we have in particular
\[(a+\lambda b)(a+\lambda b)c (a+\lambda b) = 0.\]
Developing further, we obtain
\[a^2ca+\lambda A+\lambda^2 B + \lambda^3 b^2cb = 0\]
with $A = a^2cb+abca+baca$ and $B = b^2ca+bacb+abcb$. Since $a^2ca = b^2cb = 0$ we have
\[\lambda A+\lambda^2 B = 0.\quad (\ast_1)\]
Putting $\lambda = 1$ and multiplying by $\lambda^2$ we also conclude
\[\lambda^2 A+\lambda^2 B = 0. \quad (\ast_2)\]
Our assumption that the underlying field has size $\geq 3$ implies that there exists $\lambda $ such that $\lambda - \lambda^2\neq 0$, so by subtracting $(\ast_2)$ from $(\ast_1)$ we conclude
\[a^2cb+abca+baca = 0.\quad (\dagger)\]
By setting ``$c = 1$"\footnote{This does not really make sense since a Lie algebra never has an identity element. By this, we mean re-doing the whole argument above omitting $c$ from the beginning. We will keep using this abuse of notation below.} we obtain the identity
\[a^2b+aba+ba^2 = 0\]
(which also holds in $3$-Engel Lie algebras, assuming the field is big enough). By multiplying the above on the right by $b$, we conclude 
\[a^2b^2+(ab)^2+ba^2b = 0.\quad (\ast_3)\]
Using $(\dagger )$ with $c = b$, we also get 
\[a^2b^2+ab^2a+(ba)^2= 0\]
and exchanging $a$ and $b$, we conclude 
\[b^2a^2+ba^2b+(ab)^2 = 0.\quad (\ast_4)\]
Now $(\ast_3)$ and $(\ast_4)$ yield $a^2b^2 = b^2a^2$. It does not take long to deduce from the equations and $a^2b^2 = b^2a^2$ that we also have $ba^2b = ab^2a$ and $(ab)^2 = (ba)^2$.
\begin{theorem}\label{thm:toastiesconjecture4}
    Conjecture \ref{conj:toasties1} (II) holds for $n = 4$.
\end{theorem}
\begin{proof}
Let $L$ be a 4-strong toastie algebra over a field of size $\geq 3$. As above, the strategy is to obtain a set of identities from which follows the identities $a^{3}b^{3} + b^{3}a^{3} = 0$, $ab^3a^2+ba^3b^2 = 0$, $(ab)^3+(ba)^3 = 0$ etc. Let $a,b,c_1,c_2,c_3\in L$ and $\lambda$ be a scalar. As $I(a+\lambda b)^4 = 0$ we have in particular:
\[(a+\lambda b) c_1 (a+\lambda b)c_2 (a+\lambda b)c_3 (a+\lambda b) = 0.\]
We start by developing the left hand side of the above equation.
\begin{align*}
    (a+\lambda b) c_1 (a+\lambda b)&c_2 (a+\lambda b)c_3 (a+\lambda b) = ac_1ac_2ac_3a \\
    &+ \lambda (\underbrace{bc_1ac_2ac_3a+ac_1bc_2ac_3a+ac_1ac_2bc_3a+ac_1ac_2ac_3b}_{\mbox{\quad \quad }=: A})\\
    &+ \lambda^2(\underbrace{ac_1ac_2bc_3b+ac_1bc_2ac_3b+bc_1ac_2ac_3b+bc_1ac_2bc_3a+bc_1bc_2ac_3a+ac_1bc_2bc_3a}_{\mbox{\quad \quad }=: B})\\
    &+ \lambda^3(\underbrace{ac_1bc_2bc_3b+bc_1ac_2bc_3b+bc_1bc_2ac_3b+bc_1b_2bc_3a}_{\mbox{\quad \quad }=: C})\\
    &+\lambda^4bc_1bc_2bc_3b
\end{align*}
As $ac_1ac_2ac_3a = bc_1bc_2bc_3b = 0$, we conclude that 
\[\lambda A+\lambda^2 B+\lambda^3 C = 0.\]
By hypotheses, there exist distinct nonzero scalars $\lambda_1,\lambda_2,\lambda_3$ hence we have the matrix equality
\[\begin{bmatrix}
    1 & \lambda_1 & \lambda_1^2\\
    1 & \lambda_2 & \lambda_2^2\\
    1 & \lambda_3 & \lambda_3^2
\end{bmatrix}\cdot \begin{bmatrix}
    A \\
    B \\
    C 
\end{bmatrix} = \begin{bmatrix}
    0 \\
    0 \\
    0 
\end{bmatrix}.\]
The determinant of this Vandermonde matrix is $\prod_{i<j}(\lambda_i - \lambda_j) \neq 0$ hence we deduce $A = B = C = 0$.

\noindent\textbf{Short equations.} From 
\[bc_1ac_2ac_3a+ac_1bc_2ac_3a+ac_1ac_2bc_3a+ac_1ac_2ac_3b = 0\]
we deduce, by setting $c_1 = c_2 = c_3 = 1$
\begin{align*}
    a^3b+a^2ba+aba^2+ba^3 = 0. \tag{S1}
\end{align*}
By assigning two of the $c_i$'s to $1$ and one to $b$, we deduce:
\begin{align*}
    a^3b^2+a^2b^2a+ababa+ba^2ba &= 0, \tag{S2}\\
    a^2bab+a^2b^2a+ab^2a^2+baba^2 &=0, \tag{S3}\\
    aba^2b+ababa+ab^2a^2+b^2a^3 &=0. \tag{S4}
\end{align*}
By assigning one of the $c_i$'s to $1$ and two of the $c_i$'s to $b$, we deduce:
\begin{align*}
    a^2bab^2+a^2b^3a+ab^2aba+(ba)^3 &=0, \tag{S5}\\
    aba^2b^2+abab^2a+ab^2aba+b^2a^2ba &= 0,  \tag{S6}\\
    (ab)^3+abab^2a+ab^3a^2+b^2aba^2 &=0. \tag{S7}
\end{align*}

\noindent \textbf{Long equations.} From 
\[ac_1ac_2bc_3b+ac_1bc_2ac_3b+bc_1ac_2ac_3b+bc_1ac_2bc_3a+bc_1bc_2ac_3a+ac_1bc_2bc_3a = 0\]
we deduce by setting $c_1 = c_2 = c_3 = 1$
\begin{align*}
    a^2b^2+b^2a^2+ab^2a+ba^2b+(ab)^2+(ba)^2 = 0.\tag{L1}
\end{align*}

By assigning one $c_i$ to $a$ and two $c_i$'s to $1$ we obtain 
\begin{align*}
a^3b^2+a^2bab+a^2b^2a+ba^3b+ba^2ba+baba^2 &=0, \tag{L2}\\
a^3b^2+aba^2b+ababa+ba^3b+ba^2ba+b^2a^3 &=0,   \tag{L3}\\
a^2bab+aba^2b+ab^2a^2+ba^3b+baba^2+b^2a^3 &= 0. \tag{L4}
\end{align*}
Finally, by assigning one $c_i$ to $a$, one $c_i$ to $b$ and one $c_i$ to $1$, we get
\begin{align*}
a^3b^3+a^2b^2ab+a^2b^3a+ba^2bab+ba^2b^2a+bab^2a^2 &= 0\tag{L5}\\
a^3b^3+a^2bab^2+a^2b^3a+ba^3b^2+ba^2b^2a+(ba)^3 &= 0 \tag{L6}\\
a^3b^3+aba^2b^2+abab^2a+ba^3b^2+ba^2b^2a+b^2a^2ba &= 0. \tag{L7}
\end{align*}

Now, one can check all the identities $a^3b^3+b^3a^3 = 0$, $ab^3a^2+ba^3b^2 = 0$, $(ab)^3+(ba)^3 = 0$ etc using a computational algebra software. To do so, consider the finitely presented associative algebra in two generators $a,b$ with relations $(S1) - (S7), (L1)-(L7)$ (and their ``symetrized", obtained by exchanging $a$ and $b$) and use a nilpotent quotient algorithm to compute a $7$-quotient of this finitely presented associative algebra. We did so using the Modisom GAP package \cite{ModIsom} which provides a nilpotent quotient algorithm for finitely presented associative algebras.

Surprisingly, for characteristic different from $2$, the short equations are enough to deduce the identity we seek, but in characteristic $2$ we do need to add the long equations as well. A GAP program checking all this is available on my webpage.

Nonetheless, we include a complete proof of how to deduce $a^3b^3+b^3a^3 = 0$ from the equations above in the appendix. The techniques it involves might be useful for solving the conjecture in the general case, see also the discussion there.
\end{proof}

\begin{theorem}\label{thm:toastiesconjecture5}
    Conjecture \ref{conj:toasties1} (II) holds for $n = 5$.
\end{theorem}
\begin{proof}
    The strategy here is the same as in the proof of Theorem \ref{thm:toastiesconjecture4}, only more tedious as more equations are involved. Let $L$ be a $5$-strong Lie algebra over a field of size $\geq 5$, let $a,b,c_1,c_2,c_3,c_4\in L$ and $\lambda$ a scalar. From $I(a+\lambda b)^5 = 0$ we have that
    \[(a+\lambda b)c_1(a+\lambda b)c_2(a+\lambda b)c_3(a+\lambda b)c_4(a+\lambda b) = 0.\]
    We compute the right hand side.
    \begin{align*}
        (a+\lambda b)&c_1(a+\lambda b)c_2(a+\lambda b)c_3(a+\lambda b)c_4(a+\lambda b) =\\ &ac_1ac_2ac_3ac_4a\\
        &+\lambda(\overbrace{ac_1ac_2ac_3ac_4b+ac_1ac_2ac_3bc_4a+ac_1ac_2bc_3ac_4a+ac_1bc_2ac_3ac_4a+bc_1ac_2ac_3ac_4a}^{\mbox{\quad \quad} =: A})\\
        &+\lambda^2(\overbrace{ac_1ac_2ac_3bc_4b+ac_1ac_2bc_3ac_4b+ac_1bc_2ac_3ac_4b +bc_1ac_2ac_3ac_4b}^{\mbox{\quad \quad} =:B}\\
        &\mbox{$\quad \quad \quad \quad \quad \quad \quad \quad $}+bc_1ac_2ac_3bc_4a+bc_1ac_2bc_3ac_4a+bc_1bc_2ac_3ac_4a)\\
        &+\lambda^3(\overbrace{bc_1bc_2bc_3ac_4a+bc_1bc_2ac_3bc_4a+bc_1ac_2bc_3bc_4a +ac_1bc_2bc_3bc_4a}^{\mbox{\quad \quad} =:C}\\
        &\mbox{$\quad \quad \quad \quad \quad \quad \quad \quad $}+ac_1bc_2bc_3ac_4b+ac_1bc_2ac_3bc_4b+ac_1ac_2bc_3bc_4b)\\
        &+\lambda^4(\overbrace{bc_1bc_2bc_3bc_4a+bc_1bc_2bc_3ac_4b+bc_1bc_2ac_3bc_4b+bc_1ac_2bc_3bc_4b+ac_1bc_2bc_3bc_4b}^{\mbox{\quad \quad}=: D})\\
        &+\lambda^5bc_1bc_2bc_3bc_4b
    \end{align*}
By hypotheses, $ac_1ac_2ac_3ac_4a = bc_1bc_2bc_3bc_4b = 0$ and as in the proof of Theorem \ref{thm:toastiesconjecture4}, because the underlying field is large enough, we obtain $A = B = C = D = 0$. Again, we obtain various identities by evaluating the $c_i$'s.

\noindent \textbf{Short equations.} Using the identity
\[ac_1ac_2ac_3ac_4b+ac_1ac_2ac_3bc_4a+ac_1ac_2bc_3ac_4a+ac_1bc_2ac_3ac_4a+bc_1ac_2ac_3ac_4a = 0,
\]
    
we deduce, by setting $c_1 = c_2 = c_3 = c_4 = 1$,
\begin{align*}
    a^4b+a^3ba+a^2ba^2+aba^3+ba^4 = 0. \tag{S1}
\end{align*}
By setting three of the $c_i$'s to $1$ and one of the
$c_i$'s to $b$, we get 
\begin{align*}
    a^4b^2+a^3b^2a+a^2baba+aba^2ba+ba^3ba = 0, \tag{S2}\\
    a^3bab+a^3b^2a+a^2b^2a^2+(ab)^2a^2+ba^2ba^2 = 0, \tag{S3}\\
    a^2ba^2b+a^2baba+a^2b^2a^2+ab^2a^3+baba^3 = 0, \tag{S4}\\
    aba^3b+aba^2ba+(ab)^2a^2+ab^2a^3+b^2a^4 = 0. \tag{S5}
\end{align*}
By assigning two of the $c_i$'s to $1$ and two of the
$c_i$'s to $b$, we get 
\begin{align*}
    a^3bab^2+a^3b^3+a^2b^2aba+abababa+ba^2baba = 0, \tag{S6}\\
    a^2ba^2b^2+a^2bab^2a+a^2b^2aba+ab^2a^2ba+baba^2ba = 0, \tag{S7}\\
    aba^3b^2+aba^2b^2a+abababa+ab^2a^2ba+b^2a^3ba = 0, \tag{S8}\\
    aba^2bab+aba^2b^2a+abab^2a^2+ab^2aba^2+b^2a^2ba^2 = 0, \tag{S9}\\
    ababa^2b+abababa+abab^2a^2+b^2aba^3+ab^3a^3 = 0. \tag{S10}
\end{align*}
Now assigning one of the $c_i$'s to $1$ and three of them to $b$, we get
\begin{align*}
    a^2babab^2+a^2bab^3a+a^2b^3aba+ab^2ababa+(ba)^4 = 0, \tag{S11}\\
    aba^2bab^2+aba^2b^3a+abab^2aba+ab^2ababa+b^2a^2(ba)^2 = 0, \tag{S12}\\
    (ab)^2a^2b^2+(ab)^2ab^2a+abab^2aba+ab^3a^2ba+b^2aba^2ba = 0, \tag{S13}\\
    (ab)^4+(ab)^2ab^2a+abab^3a^2+ab^3aba^2+b^2(ab)^2a^2 = 0. \tag{S14}
\end{align*}

\noindent\textbf{Long equations.} We now turn to the identity 
\begin{align*}
ac_1ac_2ac_3bc_4b+ac_1ac_2bc_3ac_4&b+ac_1bc_2ac_3ac_4b +bc_1ac_2ac_3ac_4b\\
&+bc_1ac_2ac_3bc_4a+bc_1ac_2bc_3ac_4a+bc_1bc_2ac_3ac_4a=0.
\end{align*}
We deduce, by setting $c_1 = c_2 = c_3 = c_4 = 1$,
\begin{align*}
    a^3b^2+a^2bab+aba^2b+ba^3b+ba^2ba+baba^2+b^2a^3 = 0. \tag{L1}
\end{align*}
By assigning one of the $c_i$ to $a$ and the others to $1$, we obtain:
\begin{align*}
    a^4b^2+a^3bab+a^2ba^2b+ba^4b+ba^3ba+ba^2ba^2+baba^3 = 0, \tag{L2}\\
    a^4b^2+a^3bab+aba^3b+ba^4b+ba^3ba+ba^2ba^2+b^2a^4 = 0, \tag{L3}\\
    a^4b^2+a^2ba^2b+aba^3b+ba^4b+ba^3ba+baba^3+b^2a^4 = 0, \tag{L4}\\
    a^3bab+a^2ba^2b+aba^3b+ba^4b+ba^2ba^2+baba^3+b^2a^4 = 0. \tag{L5}
\end{align*}
By setting one of the $c_i$ to $b$ and the others to $1$, we obtain:
\begin{align*}
    aba^2b^2+(ab)^3+ab^2a^2b+b^2a^3b+b^2a^2ba+b^2aba^2+b^3a^3 = 0, \tag{L6}\\
    a^2bab^2+a^2b^2ab+ab^2a^2b+baba^2b+(ba)^3+bab^2a^2+b^3a^3 = 0, \tag{L7}\\
    a^3b^3+a^2b^2ab+(ab)^3+ba^2bab+ba^2b^2a+bab^2a^2+b^2aba^2 = 0, \tag{L8}\\
    a^3b^3+a^2bab^2+aba^2b^2+ba^3b^2+ba^2b^2a+(ba)^3+b^2a^2ba = 0. \tag{L9}
\end{align*}
By assigning one of the $c_i$'s to $a$, one to $b$ and the remaining two to $1$, we obtain:
\begin{align*}
    a^3bab^2+a^3b^2ab+a^2b^2a^2b+ba^2ba^2b+ba^2(ba)^2+ba^2b^2a^2+bab^2a^3 = 0, \tag{L10}\\
    aba^3b^2+aba^2bab+ab^2a^3b+b^2a^4b+b^2a^3ba+b^2a^2ba^2+b^3a^4 = 0, \tag{L11}\\
    a^4b^3+a^3b^2ab+a^2b(ab)^2+ba^3bab+ba^3b^2a+ba^2b^2a^2+(ba)^2ba^2 = 0, \tag{L12}\\
    aba^3b^2+(ab)^2a^2b+ab^2a^3b+b^2a^4b+b^2a^3ba+b^2aba^3+b^3a^4 = 0. \tag{L13}
\end{align*}
Actually, for $(L10)-(L13)$ we only considering the cases 
\[
(c_1,c_2,c_4,c_4)\in \set{(a,b,1,1),(b,a,1,1),(a,1,b,1),(b,1,a,1)}.
\]
Now, as in the proof of Theorem \ref{thm:toastiesconjecture4}, in a computational algebra software we consider the finitely presented associative algebra with two generators and relations corresponding to equations $(S1)-(S14)$ and $(L1)-(L13)$. Then we compute a nilpotent quotient of class $10$ of this finitely presented algebra (we use the GAP package \cite{ModIsom}) and we check the identity. In characteristic $p\in \set{2,3,5,7}$ (which are the ones of interest), we actually find that 
\[a^4b^4 = 0.\]
and more generally, every monomial $p(x,y)$ of multiweight $(4,4)$ vanishes.
\end{proof}

\subsection{$\omega$-categorical $4$-strong Lie algebras}
As we already mentioned above, an $\omega$-categorical toastie algebra is $n$-strong for some $n$. In the following proof, we will also see that on the other hand, a characteristically simple $n$-strong Lie algebra is toastie.

\begin{theorem}\label{thm:4strongLAnilpotent}
    Let $L$ be an $\omega$-categorical $4$-strong Lie algebra of characteristic $p \neq 2$, then $L$ is nilpotent.
\end{theorem}
\begin{proof}
    By $\omega$-categoricity, $L$ has only finitely many characteristic ideals. Consider a maximal chain of characteristic ideals $L = J_0\supsetneq J_2\supsetneq \ldots \supsetneq J_{m+1} = 0$. Then, the quotient $J_i/J_{i+1}$ is characteristically simple and $4$-strong for each $0\leq i\leq m$. Assume that we have established that $J_i/J_{i+1}$ is nilpotent for each $0\leq i\leq m$. Then $J_{m-1}/J_m$ and $J_m$ are nilpotent, hence $J_{m-1}$ is solvable. By Corollary \ref{cor:wilsontheorem}, $J_{m-1}$ is nilpotent. An immediate induction yields that $L = J_0$ is nilpotent. Hence, we may assume that $L$ is a characteristically simple $4$-strong Lie algebra. Observe that the ideal $J$ of $L$ generated by all toasties is characteristic. Furthermore, for any $a,b\in L$, the element $[a,b^2]$ is a toastie, hence either $[a,b^2] = 0$ for all $a,b\in L$ and $L$ is $2$-Engel hence nilpotent by classical result of Higgins \cite{higginsEngel}, or $J=L$. It follows that we may assume that $L$ is a $4$-strong toastie algebra.

    Let $F$ be a finite field extension of $\F_p$ of cardinality $\geq 4$ and let $\tilde L = L\otimes _{\F_p} F$. By Lemma \ref{lm:scalarextensiontoastie}, $\tilde L$ is again a $4$-strong toastie algebra, and by Subsection \ref{subsec:interpretation}, $\tilde L$ is $\omega$-categorical. Note that here $\tilde L$ might no longer be characteristically simple. If $\tilde L$ is nilpotent, so is $L$, hence we may assume that $L$ is an $\omega$-categorical $4$-strong toastie algebra over a field of size $\geq 4$.

     We consider the set $X = \set{b^3c^3\mid b,c\in L}\seq A(L)$.
     By Theorem \ref{thm:toastiesconjecture4}, the subalgebra $B$ of $A(L)$ generated by $X$ is commutative, hence by Theorem \ref{thm:interpretationBa}, $B_a$ is an $\omega$-categorical associative algebra. As $L$ is uniformly locally nilpotent, $B_a$ is nil, hence by Cherlin's Theorem (Fact \ref{fact:cherlinomegacatnilringarenilpotent}), $B_a$ is nilpotent. We conclude that there exists $s = s(a)\in \N$ such that for all $b_1,\ldots,b_s\in L$  we have \[[ab_1^3\cdots b_s^3] = 0.\]
By $\omega$-categoricity, there is a uniform $s\in \N$ such that $[ab_1^3\ldots b_s^3] = 0$ for all $a,b_1,\ldots,b_s\in L$.
%(simply consider the family of formulas $\phi_k(x) = \exists y_1,\ldotsy_k \ [xy_1^3\ldots y_k^3] = 0$)
We now define $I_1 = \Span([ab^3]\mid a,b\in L)$, $I_2 = \Span([ab_1^3b_2^3]\mid a,b_1,b_2\in L)$,$\ldots$, and 
\[I_{s-1} = \Span([ab_1^3\cdots b_{s-1}^3]\mid a,b_1,\ldots,b_{s-1}\in L).\]
Using Lemma \ref{lm:toastielinearspanengel}, the sets $I_1,\ldots,I_{s-1}$ are ideals of $L$. Observe that $I_{s-1}$ is an $\omega$-categorical $3$-Engel Lie algebra over a field of characteristic $p\neq 2$, hence it is nilpotent: for $p\neq 5$ this is a classical result of Higgins \cite{higginsEngel} (and does not require $\omega$-categoricity), for $p = 5$ this follows from Remark \ref{rk:revisiting} below or by using \cite{delbee2024engel3}. Similarly, 
$I_{s-2}/I_{s-1}$ is also nilpotent. Then $I_{s-2}$ is solvable, so by Corollary \ref{cor:wilsontheorem}, we conclude that $I_{s-2}$ is nilpotent. An immediate iteration yields that $L$ is nilpotent. 
\end{proof}

\begin{remark}[Revisiting $3$-Engel Lie algebras of characteristic $5$]\label{rk:revisiting}
    As $2$-Engel Lie algebras in any characteristic are nilpotent, it is easy to follow the proof of Theorem \ref{thm:4strongLAnilpotent} to prove that any $\omega$-categorical $3$-strong Lie algebra is nilpotent, regardless of the characteristic.
    Gunnar Traustason \cite{TraustasonEngel3Engel41993} also proved that $3$-Engel Lie algebras of characteristic $\neq 2$ are $3$-strong, hence $\omega$-categorical $3$-Engel Lie algebras of characteristic $5$ are nilpotent, and we recover the main result of \cite{delbee2024engel3}. However the result in \cite{delbee2024engel3} is stronger because it only assumes the finiteness of $4$-types (the action of $\Aut(L)$ on $L\times L\times L\times L$ has only finitely many orbits), which is a much weaker assumption than $\omega$-categoricity. The approach given by interpreting an associative algebra and then using Cherlin's result is bound to lose control on the $n$ for which the finiteness of $n$-types is the only assumption to get nilpotency. Note that there is no reason to think that such an $n$ always exists.
\end{remark}
\subsection{The case $4$-Engel characteristic $3$} The solution to this case will be given by a theorem of Gunnar Traustason.

\begin{fact}[Traustason \cite{TraustasonEngel3Engel41993}]\label{fact:traustason}
    Let $L$ be a $4$-Engel Lie algebra in a field of characteristic $\neq 2,5$. Then $I(a)^4 = 0$ for all $a\in L$.
\end{fact}
In particular, any $4$-Engel Lie algebra in characteristic $3$ is $4$-strong. We immediately deduce our main result from Theorem \ref{thm:4strongLAnilpotent}.

\begin{corollary}\label{thm:casen=4p=3}
    Let $L$ be an $\omega$-categorical $4$-Engel Lie algebra over a field of characteristic $3$. Then $L$ is nilpotent.
\end{corollary}

\begin{corollary}\label{cor:5strongomegacategoricalnilpotent}
    Let $L$ be an $\omega$-categorical $5$-strong Lie algebra of characteristic $p\neq 2,5$. Then $L$ is nilpotent.
\end{corollary}
\begin{proof}
    First, by \cite{vaughanlee20245engelliealgebras}, $5$-Engel Lie algebras of characteristic $p>7$ are nilpotent, hence we need only to consider $p = 3$ or $p = 7$. Then, copy-paste the proof of Theorem \ref{thm:4strongLAnilpotent}, using Theorem \ref{thm:toastiesconjecture5} to reduce the problem to $4$-Engel Lie algebras. In this case, one does not need an interpretation of the associative algebra $B_a$ since by the proof of Theorem \ref{thm:toastiesconjecture5} the identity $x^4y^4 = 0$ holds in $5$-strong toastie algebras of characteristic $3$ or $7$. Then use Higgins \cite{higginsEngel} (in fact Kostrikin \cite{kostrikinAroundBurnside}) for $p = 7$ and Corollary \ref{thm:casen=4p=3} for $p = 3$.
\end{proof}

\subsection{A further reduction for $n<p$} 
\begin{theorem}\label{thm:furtherreduction}
Wilson's Conjecture holds for $n$-Engel Lie algebras of characteristic $p>n$, provided one of the following holds:
    \begin{itemize}
        \item (Assuming variant (I) of Conjecture \ref{conj:toasties1}) Every $\omega$-categorical $n$-strong $(n-1)$-Engel toastie algebra is nilpotent.
        \item (Assuming variant (III) of Conjecture \ref{conj:toasties1}) For all $n$, there exists finitely many monomials \[p_1(x_1,\ldots,x_{k_1}),\ldots,p_s(x_1,\ldots,x_{k_s})\] 
        of multiweight $(n-1,\ldots,n-1)$ such that every $\omega$-categorical $n$-strong toastie algebra satisfying the identities \[p_i(x_1,\ldots,x_{k_i}) = 0 \quad \text{for all $i = 1,\ldots,s$}\]
        is nilpotent.
    \end{itemize}
\end{theorem}

\begin{proof}
We only give the proof assuming the variant (III) of Conjecture \ref{conj:toasties1}. The other case is similar.

Assume the premises. Let $L$ be an $\omega$-categorical $n$-Engel Lie algebra with $n<p$. Let 
\[L = I_0\supsetneq \ldots \supsetneq I_{s+1} = 0\] 
be a maximal series of characteristic ideals. Each quotient $I_i/I_{i+1}$ is characteristically simple and again $n$-Engel with $n<p$ so by Wilson's Theorem (Corollary \ref{cor:wilsontheorem}) we may assume that $L$ is characteristically simple. By a famous theorem of Kostrikin \cite{kostrikinAroundBurnside}, every $n$-Engel Lie algebra over a field of characteristic $p>n$ contains a nontrivial abelian ideal, hence $L$ contains a nonzero toastie. It follows that the span of all toasties is a nontrivial characteristic ideal hence $L$ is a toastie algebra. As observed above, by $\omega$-categoricity, $L$ is $k$-strong for some $k$. Using Subsection \ref{subsec:interpretation} and Lemma \ref{lm:scalarextensiontoastie} as in the proof of Theorem \ref{thm:4strongLAnilpotent}, we may also assume that $L$ is an $k$-strong $\omega$-categorical toastie algebra over a field of size $\geq n$ (at the cost of losing characteristic simplicity).

By hypotheses, there exist $p_1(x_1,\ldots,x_{k_1}),\ldots,p_s(x_1,\ldots,x_{k_s})$ monomials of multiweight $(k-1,\ldots,k-1)$ as in the hypotheses. Under Conjecture \ref{conj:toasties1} (III), there exists $\alpha_1\in \F_p$ such that 
\[
p_1(b_1,\ldots,b_{k_{1}})p_1(c_1,\ldots,c_{k_{1}}) = \alpha_1 p_1(c_1,\ldots,c_{k_{1}})p_1(b_1,\ldots,b_{k_{1}})
\]
Let $m\in \N$ be such that $\alpha_1^m = 1$. Let $\vec x_1,\ldots,\vec x_m$ be $m$ tuples of length $k_1$ and set $\vec x := (\vec x_1,\ldots,\vec x_m)$. We now define
\[q(\vec x) := p_1(\vec x_1) \cdots p_1(\vec x_{m})\]
Then $q(\vec b)\cdot p_1(c_1,\ldots,c_{k_1}) =  p_1(c_1,\ldots,c_{k_1})q(\vec b)$ holds for all $\vec b,c_i$ from $L$ so that for all $\vec b,\vec c\in L^{mk_1}$ we have 
\[q(\vec b)q(\vec c) = q(\vec c)q(\vec b).\]
This implies that the associative algebra $B\seq A(L)$ generated by $X = \set{q(\vec b)\mid \vec b\in L^{mk_1}}\seq A(L)$ is commutative. For any $a\in L$, using Theorem \ref{thm:interpretationBa} and Cherlin's Theorem (Fact \ref{fact:cherlinomegacatnilringarenilpotent}), we conclude that $B_a$ is nilpotent. It follows that there exists $t\in \N$ such that 
    \[[aq(\vec b_1)\cdots q(\vec b_t)] = 0\]
    for all $a,\vec b_1,\ldots,\vec b_t\in L^{mk_1}$. Then, for $s = tm$, we have  
    \[[ap_1(\vec b_1)\cdots p_1(\vec b_s)] = 0\]
    for all $\vec b_1,\ldots,\vec b_s\in L^{k_1}$.
    
    Consider $I_i = \Span([ap_1(\vec b_1)\cdots p_1(\vec b_i)]\mid a,\vec b_1,\ldots,\vec b_i\in L)$. By an easy generalization of Lemma \ref{lm:toastielinearspanengel} (see Remark \ref{rk:moretedious}), $I_1,\ldots,I_s$ are ideals of $L$. Note that each element of the form $[ap_1(\vec b_1)\cdots p_1(\vec b_i)]$ is a toastie in $L$, so that $I_i$ and $I_i/I_{i+1}$ are toastie algebras. The ideal generated by an elements $u$ of $I_i$ computed in $I_i$ also satisfy $I^k(u) = 0$ hence $I_i$ and $I_i/I_{i+1}$ are $k$-strong toastie algebras. Each quotient $I_i/I_{i+1}$ is $\omega$-categorical, $n$-strong toastie algebra and satisfies the identity
    \[[ap_1(\vec b)] = 0.\]

    Working now inside each quotient $L_i = I_i/I_{i+1}$, we may re-do the argument above to decompose $L_i$ in finitely many sections for which $p_2(\vec b) = 0$. We may iterate this procedure so that, using exhaustively Corollary \ref{cor:wilsontheorem}, we have reduced the question of nilpotency of $L$ to the nilpotency of $\omega$-categorical $n$-strong toastie algebras which satisfy $p_i(x_1,\ldots,x_{k_i}) = 0$, for all $i = 1,\ldots,s$. 
\end{proof}

\section{Appendix}
\subsection{A direct proof of Theorem \ref{thm:toastiesconjecture4}} We prove that $4$-strong Lie algebras satisfy the identity $a^3b^3+b^3a^3 = 0$. We already proved that $L$ satisfies the following set of short equations:

\begin{align*}
    a^3b+a^2ba+aba^2+ba^3 &= 0, \tag{S1}\\
    a^3b^2+a^2b^2a+ababa+ba^2ba &= 0, \tag{S2}\\
    a^2bab+a^2b^2a+ab^2a^2+baba^2 &=0, \tag{S3}\\
    aba^2b+ababa+ab^2a^2+b^2a^3 &=0, \tag{S4}\\
    a^2bab^2+a^2b^3a+ab^2aba+(ba)^3 &=0, \tag{S5}\\
    aba^2b^2+abab^2a+ab^2aba+b^2a^2ba &=0, \tag{S6}\\
    (ab)^3+abab^2a+ab^3a^2+b^2aba^2 &=0, \tag{S7}
\end{align*}
and the following long equations:
\begin{align*}
a^2b^2+b^2a^2+ab^2a+ba^2b+(ab)^2+(ba)^2 &= 0, \tag{L1}\\
a^3b^2+a^2bab+a^2b^2a+ba^3b+ba^2ba+baba^2 &=0, \tag{L2}\\
a^3b^2+aba^2b+ababa+ba^3b+ba^2ba+b^2a^3 &=0, \tag{L3}\\
a^2bab+aba^2b+ab^2a^2+ba^3b+baba^2+b^2a^3 &= 0, \tag{L4}\\
a^3b^3+a^2b^2ab+a^2b^3a+ba^2bab+ba^2b^2a+bab^2a^2 &= 0, \tag{L5}\\
a^3b^3+a^2bab^2+a^2b^3a+ba^3b^2+ba^2b^2a+(ba)^3 &= 0,  \tag{L6}\\
a^3b^3+aba^2b^2+abab^2a+ba^3b^2+ba^2b^2a+b^2a^2ba &= 0. \tag{L7}
\end{align*}

We start by listing a few homogeneous equations that follow from the above.
\begin{align*}
    a^3b^3+a^2b^2ab+(ab)^3+ba^2bab &= 0, \tag{S8}\\
    b^2a^3b+b^2a^2ba+b^2aba^2+b^3a^3 &= 0, \tag{S9}\\
    baba^2b+(ba)^3+bab^2a^2+b^3a^3 &= 0, \tag{S10}\\
    ba^3b^2+ba^2b^2a+(ba)^3+b^2a^2ba &= 0, \tag{S11}\\
    ba^3b^2+ba^2bab+baba^2b+b^2a^3b &= 0, \tag{S12}\\
    a^2bab^2+a^2b^2ab+ab^2a^2b+baba^2b &=0, \tag{S13}\\
    a^3b^3+a^2bab^2+aba^2b^2+ba^3b^2 &= 0, \tag{S14}\\
    aba^2b^2+ab^3a^2+abab^2a+ab^2a^2b+(ab)^3+ab^2aba &= 0, \tag{L8}\\
    a^2b^2ab+b^2a^3b+ab^2a^2b+ba^2bab+(ab)^3+baba^2b &= 0, \tag{L9}\\
    a^3b^3+ab^2a^2b+a^2b^2ab+aba^2b^2+a^2bab^2+(ab)^3& =0. \tag{L10}
\end{align*}
$(S8)$ is obtained by multiplying $(S2)$ on the right by $b$.
$(S9)$ is obtained by multiplying $(S1)$ on the left by $b^2$.
$(S10)$ is obtained by multiplying $(S4)$ by $b$ on the left.
$(S11)$ is obtained by multiplying $(S2)$ by $b$ on the left.
$(S12)$ is obtained by multiplying $(S1)$ by $b$ on the left and on the right.
$(S13)$ is obtained by multiplying $(S3)$ by $b$ on the right.
$(S14)$ is obtained by multiplying $(S1)$ by $b^2$ on the right.
$(L8)$ is obtained by multiplying $(L1)$ by $ab$ on the left.
$(L9)$ is obtained by multiplying $(L1)$ by $ab$ on the right.
$(L10)$ is obtained by multiplying $(L1)$ by $a$ on the left and $b$ on the right.

For a polynomial $p(a,b)$ we define the predicate
\[\Asym(p(a,b)) \iff p(a,b) = -p(b,a).\]
We also define \[\overline{p(a,b)} := p(b,a)\] and we notice the following elementary rules, for $p,q$ polynomials in $(a,b)$:
\begin{align*}
    &\Asym(p+q) \iff \Asym(p+\overline{q}), \tag{Rule1}\\
    &\Asym(p) \text{ and } p+q = 0 \implies \Asym(q). \tag{Rule2}
\end{align*}

The number of monomials of multiweight $(3,3)$ in $(a,b)$ is $6$ choose $3$ which is $20$. As $\Asym(p)$ is equivalent to $\Asym(\bar p)$, our final goal is to establish $\Asym(p)$ for $p$ that ranges in
\[\set{a^3b^3, (ab)^3, a^2b^3a, ab^3a^2, ab^2aba, aba^2b, a^2b^2ab, aba^2b^2, ab^2a^2b, a^2bab^2}.\]

We first establish that 
\[\Asym(a^3b^3+a^2b^2ab). \tag{$\star_1$}\]
Consider $(S5)$ with $a$ and $b$ exchanged and $(S8)$:
\begin{align*}
b^2aba^2+b^2a^3b+ba^2bab+(ab)^3 &=0,\\
    a^3b^3+a^2b^2ab+(ab)^3+ba^2bab &= 0.
\end{align*}
By substracting the two equations above, we obtain
\[a^3b^3+a^2b^2ab = b^2aba^2+b^2a^3b.\]
Replacing $b^2aba^2+b^2a^3b$ within equation $(S9)$ 
\begin{align*}
b^2a^3b+b^2a^2ba+b^2aba^2+b^3a^3 &= 0
\end{align*}
we conclude
\[a^3b^3+a^2b^2ab+b^3a^3+b^2a^2ba = 0\]
hence $(\star_1)$ holds.

We apply Rule2 with $(\star_1)$ and equation $(S8)$ to obtain $\Asym((ab)^3+ba^2bab)$ and by Rule1 we get
\[\Asym((ab)^3+ab^2aba). \tag{$\star_2$}\]
Using Rule1 we also have $\Asym((ba)^3+ab^2aba)$ hence using Rule2 with equation $(S5)$ we obtain 
\[\Asym(a^2b^3a+a^2bab^2).\tag{$\star_3$}\]
Consider now equations $(S10)$ and $(S8)$
\begin{align*}
    a^3b^3+a^2b^2ab+(ab)^3+ba^2bab &= 0 \tag{S8}\\
    baba^2b+(ba)^3+bab^2a^2+b^3a^3 &= 0 \tag{S10}
\end{align*}
and add them together:
\[\underbrace{a^3b^3+b^3a^3+a^2b^2ab}_{ = -b^2a^2ba\text{ by }(\star_1)}+\underbrace{(ab)^3+(ba)^3+ba^2bab}_{ = -ab^2aba \text{ by } (\star_2)}+baba^2b+bab^2a^2 = 0.\]
We conclude 
\[baba^2b+bab^2a^2 = ab^2aba+b^2a^2ba.\]
Replacing $ab^2aba+b^2a^2ba$ in $(S6)$, we conclude:
\[\Asym(aba^2b^2+abab^2a).\tag{$\star_4$}\]
Using Rule2 with $(S6)$ and $(\star_4)$ we obtain $\Asym(ab^2aba+b^2a^2ba)$ and with Rule1 we get
\[\Asym(a^2b^2ab+ab^2aba).\tag{$\star_5$}\]
Using Rule2 with $(S8)$ and $(\star_5)$ we conclude
\[\Asym(a^3b^3+(ab)^3)\tag{$\star_6$}.\]

Equation $(S6)$ yields
\[aba^2b^2+abab^2a+ab^2aba = -b^2a^2ba.\]
Replacing the term $aba^2b^2+abab^2a+ab^2aba$ in equation $(L8)$:
\[aba^2b^2+ab^3a^2+abab^2a+ab^2a^2b+(ab)^3+ab^2aba = 0,\]
we obtain 
\[ab^3a^2+ab^2a^2b+(ab)^3 = b^2a^2ba.\]
Exchanging $a$ and $b$ in equations $(S11)$ yields
\[\underbrace{ab^3a^2+ab^2a^2b+(ab)^3}_{ = b^2a^2ba}+a^2b^2ab = 0\]
whence
\[\Asym(a^2b^2ab).\tag{$\pumpkin_1$}\]
Using Rule2 with $(\pumpkin_1)$ and $(\star_1)$, we get 
\[\Asym(a^3b^3).\tag{$\pumpkin_2$}\]
Again, Rule2 with $(\pumpkin_2)$ and $(\star_6)$ yields 
\[\Asym((ab)^3)\tag{$\pumpkin_3$}\]
which, together with $(\star_2)$ yields
\[\Asym(ab^2aba). \tag{$\pumpkin_4$}\]
Using again equation $(S11)$ exchanging $a$ and $b$, we obtain 
\[ab^2a^2b+(ab)^3+a^2b^2ab = -ab^3a^2.\]
Replacing the term $ab^2a^2b+(ab)^3+a^2b^2ab$ in $(L9)$ gives 
\[ab^3a^2 = b^2a^3b+ba^2bab+baba^2b\]
Replacing $b^2a^3b+ba^2bab+baba^2b$ by $ab^3a^2$ in $(S12)$:
\[ba^3b^2+ba^2bab+baba^2b+b^2a^3b = 0\]
yields
\[\Asym(ab^3a^2). \tag{$\pumpkin_5$}\]
Clearly, $(\pumpkin_1), (\pumpkin_3)$ and $ (\pumpkin_5)$ gives $\Asym(b^2a^2ba+(ba)^3+ba^3b^2)$ hence by Rule2 with $(S11)$ we obtain:
\[\Asym(ab^2a^2b). \tag{$\pumpkin_6$}\]
Considering equation $(S10)$ with $a$ and $b$ exchanged, we get 
\[a^3b^3+aba^2b^2+(ab)^3 = -abab^2a.\]
Replacing $a^3b^3+aba^2b^2+(ab)^3$ by $-abab^2a$ in $(L10)$:
\[
a^3b^3+ab^2a^2b+a^2b^2ab+aba^2b^2+a^2bab^2+(ab)^3=0,
\]
gives 
\[abab^2a = ab^2a^2b+a^2b^2ab+a^2bab^2.\]
Now replacing $ab^2a^2b+a^2b^2ab+a^2bab^2$ by $abab^2a$ in $(S13)$ 
\[a^2bab^2+a^2b^2ab+ab^2a^2b+baba^2b = 0\]
yields 
\[\Asym(abab^2a). \tag{$\pumpkin_7$}\]
Applying Rule2 with $(\pumpkin_7)$ and $(\star_4)$ yields
\[\Asym(aba^2b^2). \tag{$\pumpkin_8$}\]
By $(\pumpkin_2), (\pumpkin_5)$ and $(\pumpkin_8)$, we have $\Asym(a^3b^3+aba^2b^2+ba^3b^2)$, hence by Rule2 with equation $(S14)$:
\[a^3b^3+a^2bab^2+aba^2b^2+ba^3b^2 = 0,\]
we obtain 
\[\Asym(a^2bab^2). \tag{$\pumpkin_9$}\]
A last use of Rule2 with $(\pumpkin_9)$ and $(\star_3)$ yields
\[\Asym(a^2b^3a). \tag{$\pumpkin_{10}$}\]
Observe that, in the end, we only used equations $(S1)-(S6)$ and $(L1)$.

\subsection{Remarks for $5$-strong Lie algebras, and towards a generalization of the argument} For $5$-strong Lie algebras, a direct proof of Theorem \ref{thm:toastiesconjecture5} is possible via the same strategy, although it is much more tedious. Nonetheless, an important subtlety should be pointed out. The proof proceeds as follows: first, one should consider the full set of homogenizations of all the equations from Theorem \ref{thm:toastiesconjecture5} to sums of monomials of weight $(4,4)$. This gives tens of equations. Then, one should consider an ``operator of symmetry" under which the set of equations is symmetric, and which allows to find rules such as in Rule1 and Rule2. In the case of $4$-strong Lie algebras, the operator is clearly $p\mapsto -\overline{p}$. It appears that this choice of operator of symmetry is not pertinent in the case of $5$-strong Lie algebras. In fact, the operator which is meaningful in that case is the ``mirror" operator, which gives to a monomial its mirror: $a^3b^4a \mapsto ab^4a^3$, $aba^2b^3a\mapsto ab^3a^2ba$, $a^3b^3ab\mapsto bab^3a^3$ etc. It turns out that the set of homogeneous equations is closed under the mirror operator, and a similar argument as above can be carried out.

Consider the ring $R = \F_p[X]/\vect{X^2-1} = \F_p[\sigma]$, where $\sigma^2 = 1$. It is easy to formalize the problem above as a question in linear algebra over the ring $R$: we are given a partition $S = S_0\cup S_1$ of the set of all homogeneous monomials of weight $(k-1,k-1)$ and a symmetric operator $\sigma: S\to S$ which maps $S_0$ to $S_1$ and $S_1$ to $S_0$, with $\sigma^2 = \Id$. The set of homogeneous equations gives a matrix $M$ with entries in $\set{0,1,\sigma}$ such that if $\vec v$ is a vector consisting of all elements of $S_0$, the system of homogeneous equations is equivalent to $M\vec v = 0$. The condition $I^k(x) = 0$ should imply that this matrix has enough rows to ensure a transformation of the matrix into one of the form 
\[
\begin{bmatrix} 
    1-\sigma & * & \dots  & *\\
    0 & 1-\sigma & &\\
    \vdots & 0 &\ddots & \vdots\\
    0 & \dots & & 1-\sigma 
\end{bmatrix}
\]
which would give a proof of the conjecture.

\bibliographystyle{plain}
\bibliography{biblio.bib}{}

\end{document}